\documentclass{ejm2}

\usepackage{hyperref}
\usepackage{amsmath}
\usepackage{color}
\usepackage{psfrag}
\usepackage{graphicx}
\usepackage[shortlabels]{enumitem}
\ifprodtf \else
  \checkfont{eurm10}
  \iffontfound
    \IfFileExists{upmath.sty}
      {\typeout{^^JFound AMS Euler Roman fonts on the system,
                   using the 'upmath' package.^^J}%
       \usepackage{upmath}}
      {\typeout{^^JFound AMS Euler Roman fonts on the system, but you
                   don't seem to have the}%
       \typeout{'upmath' package installed. EJM.cls can take advantage
                 of these fonts,^^Jif you use 'upmath' package.^^J}%
       \providecommand\mathup[1]{##1}%
       \providecommand\mathbup[1]{##1}%
      }
  \else
    \providecommand\mathup[1]{#1}%
    \providecommand\mathbup[1]{#1}%
  \fi
\fi

\ifprodtf \else
  \checkfont{msam10}
  \iffontfound
    \IfFileExists{amssymb.sty}
      {\typeout{^^JFound AMS Symbol fonts on the system, using the
                'amssymb' package.^^J}%
       \usepackage{amssymb}%
       \let\le=\leqslant  \let\leq=\leqslant
       \let\ge=\geqslant  \let\geq=\geqslant
      }{}
  \fi
\fi

\ifprodtf \else
  \IfFileExists{amsbsy.sty}
    {\typeout{^^JFound the 'amsbsy' package on the system, using it.^^J}%
     \usepackage{amsbsy}}
    {\providecommand\boldsymbol[1]{\mbox{\boldmath $##1$}}}
\fi

\newsavebox{\astrutbox}
\sbox{\astrutbox}{\rule[-5pt]{0pt}{20pt}}

\newcommand{\x}{{\bf x}}
\newcommand{\vx}{\vec x}
\newcommand{\X}{{\bf X}}
\newcommand{\bfn}{{\bf n}}
\newcommand{\ud}{\mathrm{d}}

\newcommand{\pe}{p^\epsilon}

\newcommand{\phie}{\phi^\epsilon}
\newcommand{\Ee}{E_N^\epsilon}
\newcommand{\Eed}{E_2^\epsilon}
\newcommand{\EN}{\mathcal E_N^\epsilon}
\newcommand{\Pout}{P_\text{out}}

\newcommand\etal{\mbox{\textit{et al. }}}

\newcommand\eg{e.g.\ }

\newtheorem{theorem}{Theorem}[section]
\newdefinition{definition}[theorem]{Definition}
\newdefinition{example}[theorem]{Example}
\newdefinition{remark}[theorem]{Remark}
\newdefinition{assum}[theorem]{Assumptions}
\newtheorem{lemma}{Lemma}
\newtheorem{corollary}{Corollary}
\newtheorem{prop}{Proposition}

\title[Coarse graining of a Fokker--Planck equation]{Coarse graining of a Fokker--Planck equation with excluded volume effects preserving the gradient-flow structure}

\author[M. Bruna et al.]{%
  M.\ns B\ls R\ls U\ls N\ls A$\,^1$,\ns
  M.\ns B\ls U\ls R\ls G\ls E\ls R$\,^2$\ns
\and
  J.\ns A.\ns C\ls A\ls R\ls R\ls I\ls L\ls L\ls O$\,^3$
}

\affiliation{%
  $^1\,$Department of Applied Mathematics and Theoretical Physics, University of Cambridge, Cambridge CB3 0WA, UK \textup{(\texttt{bruna@maths.cam.ac.uk})}\\
  $^2\,$Friedrich-Alexander-Universit\"at Erlangen-N\"urnberg Department Mathematik, Cauerstrasse 11, 91058 Erlangen, Germany \textup{(\texttt{martin.burger@fau.de})}\\
  $^3\,$Mathematical Institute, University of Oxford, Oxford OX2 6GG, UK \textup{(\texttt{carrillo@maths.ox.ac.uk})} }

\date{\today}

\graphicspath{{./figures/}}

\begin{document}

\maketitle

\begin{abstract}
The propagation of gradient flow structures from microscopic to macroscopic models is a topic of high current interest. In this paper we discuss this propagation in a model for the diffusion of particles interacting via hard-core exclusion or short-range repulsive potentials. We formulate the microscopic model as a high-dimensional gradient flow in the Wasserstein metric for an appropriate  free-energy functional. Then we use the JKO approach to identify the asymptotics of the metric and the  free-energy functional beyond the lowest order for single particle densities in the limit of small particle volumes by matched asymptotic expansions. While we use a propagation of chaos assumption at far distances, we consider correlations at small distance in the expansion. In this way we obtain a clear picture of the emergence of a macroscopic gradient structure incorporating corrections in the free energy functional due to the volume exclusion.
\end{abstract}


\




\section{Introduction} \label{sec:intro}

An interesting feature of many partial differential equations (PDEs) describing dissipative mechanisms in particle systems is that they can be seen as gradient flows (or steepest descents) of an associated  free-energy functional. This is the case of the linear Fokker--Planck equation \cite{Jordan:1998wa}, which describes the evolution of the probability of one or many Brownian independent particles, and many other nonlinear Fokker--Planck equations including nonlinear diffusions and McKean--Vlasov like equations \cite{AGS,CMV2003,Otto,santambrogio2015optimal,villani2003topics}. For example, if we consider $N$ Brownian particles moving under an external potential $V(\x)$, their evolution can be described by the following stochastic differential equation (SDE):
\begin{equation} 
\label{hsde}
\ud \X_i(t) =  \sqrt{2} \, \ud{\bf W}_i(t) - \nabla V_\x ( \X_i (t) ) \ud t, \qquad 1 \le i \le N,
\end{equation}
where ${\bf W}_i(t)$ are independent Brownian motions. 
The set of $N$ particles can equivalently be described by a Fokker--Planck PDE for its joint probability density $P ({\vx}, t)$, where $\vx = (\x_1, \dots, \x_N)$:
\begin{align}
\label{fp_eq}
\frac{\partial P}{\partial t} (\vx, t) &= \nabla_{\vx} \cdot \left [  \nabla_{\vx} P + \nabla_{\vx} V_N (\vx) P \right ],
\end{align}
and $\nabla _{\vx}$ and $\nabla _{\vx} \, \cdot$ respectively stand for the gradient and divergence operators with respect to the $N$-particle position vector $\vx$  and $V_N (\vx) = \sum_{i=1}^N V(\x_i)$. The Fokker--Planck equation \eqref{fp_eq} can be seen as a gradient flow 
\begin{equation*}
\frac{\partial P}{\partial t} (\vx, t)  =  \nabla_{\vx} \cdot \left (P  \nabla_{\vx} \frac{\delta \mathcal E_N }{\delta P} \right),
\end{equation*}
with respect to the Wasserstein metric and the free energy
\begin{equation}
\label{entropy}
\mathcal E_N(P)  = \int  \left[ P(\vx, t) \log P(\vx,t) + V_N (\vx) P  \right] \, \ud \vx.
\end{equation}

The connections between \eqref{hsde}, \eqref{fp_eq} and \eqref{entropy} are well understood in the case of noninteracting particles, where essentially the macroscopic limit of a set of $N$ particles coincides with the case of a single Brownian particle \cite{Jordan:1998wa}, leading to the linear Fokker--Planck equation for the one-particle probability density $p(\x,t)$
\begin{align}
\label{fp_one_eq}
\frac{\partial p}{\partial t} (\x, t) &= \nabla_{\x} \cdot \left [   \nabla_{\x} p + \nabla_{\x}  V ({\x})p \right ], 
\end{align}
In this paper we are interested in the connections between these objects when we consider interacting Brownian particles. 

Having interactions makes the coarse-graining procedure, going from $N$ particles to one, highly non-trivial. In particular, the Fokker--Planck equation for the one-particle marginal density $p(\x,t) = \int P(\vx,t) \delta(\x-\x_1) \ud \vx$ becomes in general coupled to higher-order marginals, leading to a BBGKY-type hierarchy, and its relation to the $N$-particle probability density becomes much more complicated due to correlations between particles. The particular assumptions on the interactions are crucial in order to derive the evolution of the one-particle probability density. 

Consider a set of $N$ pairwise interacting particles in a bounded domain $\Omega \subset \mathbb R^d$ with an interaction potential $u$:
\begin{equation} 
\label{ssde0}
\ud \X_i(t) =  \sqrt{2} \, \ud{\bf W}_i(t) - \nabla_\x V ( \X_i (t) ) \ud t  - \chi \sum_{j \ne i}  \nabla_\x u((\X_i(t) - \X_j(t))/\ell) \ud t,
\end{equation}
for $1\le i \le N$, where $\chi$ and $\ell$ represent the strength and the range of the potential $u$, respectively. Depending on $\chi$ and $\ell$, one expects different limit equations \cite{Bodnar:2005kv}. When the interactions are long range ($\ell \sim 1$), then one particle interacts on average with an order $N$ particles. The mean-field approximation $P(\vx,t) = \prod_{i = 1}^N p(\x_i,t)$ leads to the nonlinear McKean--Vlasov equation 
\begin{equation} 
\label{fp_mfa0}
\frac{\partial p}{\partial t} (\x,t) = \nabla_{{\bf  x}}  \cdot \left[ \nabla_{{\bf  x}} \,  p  + \nabla_{\x} V(\x) p  +  \chi (N-1) p \! \int_{\Omega} p( {\bf y}, t) \nabla_{\x} u(\x - {\bf y})   \,  \ud {\bf y}   \right],
 \end{equation}
for $\x, {\bf y} \in \Omega$. The approximation can be made rigorous taking the so-called mean-field scaling $\chi = 1/N$ and taking the  limit $N\to \infty$ so that the $N$-dependence in \eqref{fp_mfa0} drops out.
Rigorous proofs of the mean-field limit \eqref{fp_mfa0} typically require the potential $u$ to be Lipschitz or other less restrictive assumptions, see 
\cite{BCC,CDP,Jabin:2017fb,oelschlager1984martingale,Szn} and references therein, and only recently singular potentials such as the Coulomb or Newtonian potential have been included \cite{MR4018082,MR3858403}. The mean-field limit relies on each particle interacting on average with all the other particles, and it is therefore, not suitable for repulsive short-range interactions ($\ell \ll 1$), where typically a particle only interacts with close neighbours but when it does, the interaction is strong. In this paper, we are interested in the regime $\ell = \epsilon \ll 1$, $\chi = 1$ and $N \epsilon^d \ll 1$. 
 Using the method of matched asymptotic expansions in this limit one obtains a nonlinear correction term in the development for small $\epsilon$ and $N$ fixed, leading to the nonlinear Fokker--Planck equation \cite{Bruna:2017vr}
 \begin{align}
\label{fp_one_eq0}
\frac{\partial p}{\partial t} (\x, t) &= \nabla_{\x} \cdot \left \{  \left[1 + \alpha_u (N-1) \epsilon^d p \right] \nabla_{\x} p + \nabla_{\x}  V ({\x})p \right \}, 
\end{align}
where
\begin{equation} \label{alpha_u0}
\alpha_u =  \int_{\mathbb R^d}  \left ( 1 - e^{- u( \x )} \right)  \ud  \x,
\end{equation}
is independent of $\epsilon$. 
Here $\pe$ is used to indicate the asymptotic approximation of $p$ for small $\epsilon$. Thus, we see that interactions introduce a nonlinear diffusion term into the macroscopic Fokker--Planck \eqref{fp_one_eq}. Interestingly though, \eqref{fp_one_eq0} preserves the gradient-flow structure of the original microscopic Fokker--Planck \eqref{fp_eq}, with the following  free-energy 
\begin{equation}
\label{entropy_1}
\Ee (p) = \int_{\Omega} \left[p \log p + \frac{1}{2}\alpha_u (N-1) \epsilon^d p^2 + V({\x_1})p \right]  \, \ud {\x_1},
\end{equation}
see for instance \cite{CMV2003,Otto,villani2003topics}. The same applies when, instead of soft interactions, one considers hard core interactions between particles (hard spheres of diameter $\epsilon$). In that case, interactions do not appear in \eqref{ssde0} but as boundary conditions on a perforated domain, that is, $\| \X_i (t)- \X_j(t)\| = \epsilon$ and the coefficient is $\alpha_u = V_d(1)$, the volume of the unit ball in $\mathbb R^d$. Note that both the microscopic and the macroscopic densities $P$ and $p$ depend on the small parameter $\epsilon$. We will not make this dependency explicit for notational simplicity. 
We discuss this further in Subsection \ref{sec:hs}.
Other scaling limits are possible \cite{Bodnar:2005kv,BCM,Oelschlager:1985kv} and they will be discussed in Section \ref{sec:variational} below.

The aim of this work is to study what happens to the gradient flow structure and free-energy of a particle-based model when coarse-graining. In the examples above it has been shown by mean field or matched asymptotic expansions directly on the PDE level that the macroscopic model preserves the structure, and that interactions appeared as a quadratic term in the  free-energy. But, can we recover this information from the variational viewpoint during the coarse graining procedure? Understanding this point would be potentially useful in the description of generalized models \eqref{hsde} for non-identical particles, for which the macroscopic model is a cross-diffusion PDE system \cite{Bruna:2012wu} without a gradient flow structure, at least not in general and in the standard sense \cite{Bruna:2016cm}. 

The idea of coarse-graining at the level of the variational Fokker--Planck equation or the free-energy has already been considered for other systems, such as discrete simple-exclusion processes and mean-field interactions, using the theory of large deviations or Gamma-convergence, see for example \cite{Adams:2013gi, Adams:2011kh,Berman:2019jw,Fathi:2016gl, Kaiser:2018ti}.
This also embeds into a more abstract setting of evolutionary convergence of gradient flows, see \cite{arnrich2012passing,mielke2016evolutionary,mielke2019coarse,sandier2004gamma,serfaty2011gamma}.
All these papers are working on the lowest-order limit, while we seek to derive a first-order expansion in terms of the small volume of particles (note that the lowest order in our case is a linear Fokker--Planck equation for single particles that can be obtained easily).
Here we propose to work with the variational scheme, also called the JKO scheme, leading to the Fokker--Planck equation \cite{Jordan:1998wa} and use the method of matched asymptotics at that level to obtain a macroscopic variational scheme, thus ensuring the preservation of the gradient-flow structure. In this paper we focus on the simple example with identical interacting particles, for which we have a gradient flow structure at the macroscopic level as discussed above.

The paper is organized as follows. In Section \ref{sec:outline} we introduce the key definitions and give an outline of the main results.
We first introduce the variational formulation of the steepest descent at the macroscopic and microscopic levels together with their optimality conditions in Section \ref{sec:variational}. We discuss some aspects such as uniqueness of the variational formulations and their linearisation. Section \ref{sec:onespecies} is the core of this work devoted to the strategy of matched asymptotic expansions at the level of the optimality conditions for the variational schemes. This is all done in the case of soft particles while a final subsection deals with the hard-sphere case. Section \ref{sec:numerics} illustrates these results with numerical experiments emphasizing the free energy comparisons between microscopic and macroscopic simulations. We conclude with a discussion of our results and future work in Section \ref{sec:conclusion}.

\section{Outline of the results} \label{sec:outline}

We consider $N$ identical particles evolving according to \eqref{ssde0} with $\chi = 1$ and $\ell = \epsilon \ll 1$ in a bounded domain $\Omega \subset \mathbb R^d$ with $|\Omega| = 1$ and no-flux boundary conditions on $\partial \Omega$. The main result of this work is obtained for strong repulsive short-ranged interactions. We shall indeed assume throughout this text the following properties for $u$ (a \emph{short-range potential}).

\begin{assum}[Short-range potential] \label{assum_u} The potential $u : \mathbb R ^d \to \mathbb R$ is a radial, nonnegative function whose gradient is locally Lipschitz outside the origin. Moreover, it is assumed that $u$ is unbounded near zero, goes to zero at $\infty$ with bounded derivatives and $u = O(r^{-(d+\delta)})$ for some $\delta>0$ as $r\to \infty$. 	Note that in particular $u$ may not be integrable. 
\end{assum}

\begin{assum}[Asymptotic regime] \label{assum_regime} We assume that the range of the potential $\epsilon$ and the volume fraction $\phi$ are small parameters, $\epsilon \ll 1$ and  $ \phi = N \epsilon^d  \ll 1$. The number of particles $N$ can either be fixed and finite, or consider the case $N \to \infty$ at the appropriate rate as $\epsilon \to 0$ so that $\phi$ remains finite.
\end{assum}

\begin{remark}
	We note that our regime corresponds to a higher density than in the Boltzmann scaling of $\epsilon \to 0$ and $N\to \infty$ with $N \epsilon^{d-1} = k$ finite. The volume fraction in the Boltzmann limit goes to zero, so the first correction we are calculating would vanish in that regime.
\end{remark}

The joint probability density $P(\vx,t)$, $\vx = (\x_1, \dots, \x_N)$, of $N$ particles evolving according to \eqref{ssde0} satisfies the following problem
\begin{subequations}
\label{fps}
\begin{alignat}{3}
\label{fps_eq}
\frac{\partial P}{\partial t} &= \nabla_{\vx} \cdot \left [  \nabla_{\vx} P + \nabla_{\vx} V_N (\vx) P + \nabla_{\vx} U_N^\epsilon (\vx) P \right ], & \qquad &\vx \in \Omega^N, t>0,\\
0 & =  \vec n \cdot \left [  \nabla_{\vx} P + \nabla_{\vx} V_N (\vx) P + \nabla_{\vx} U_N^\epsilon (\vx ) P \right ], &  &\vx \in \partial \Omega^N, t>0,
\end{alignat}
where $\vec n$ is the outward normal vector on $\partial \Omega^N$, $V_N(\vx) = \sum_{i=1}^N V(\x_i)$ as before, and $U_N$ is the total interaction potential of the system
\begin{equation}
\label{interact_pot}
U_N^\epsilon(\vx) = \sum_{i=1}^N \sum_{j>i}^N u\left(\frac{ \x_i - \x_j}{\epsilon}  \right).
\end{equation}
The corresponding  free-energy is given by
\begin{equation}
\label{entropys}
\EN(P)  = \int_{\Omega^N}  \left[ P(\vx, t) \log P(\vx,t) + V_N (\vx) P  + U_N^\epsilon(\vx) P \right] \, \ud \vx.
\end{equation}
\end{subequations}

\begin{assum}[Initial conditions] \label{assum_initial}  
Throughout this work we will consider the initial positions $\X_i(0)$ to be random, indistinguishable, and identically distributed, with
\[
\text{Law}\big(\X_1(0), \X_2(0), \dots, \X_N(0)\big) = P_0(\vx).
\]
This implies that $P_0$ is invariant to permutations of the particle labels. 
Moreover, we assume their initial law to behave at small distances ($\|\X_i-\X_j\| \sim \epsilon$) such that \eqref{fps_eq} has sufficiently regular solutions and to behave like a chaotic ensemble at larger distances  ($\|\X_i-\X_j\| \gg \epsilon$). We will be more precise in the Subsection \ref{sec:mae_soft} (see Assumptions \ref{assum_initial_refined}). 
\end{assum}

As we will recall in the next section, there exists a natural variational scheme associated to the Fokker--Planck equation \eqref{fps}, consisting of the following time-discrete approximation \cite{Jordan:1998wa,Otto}. 
Let $\bar P_{k}(\vx)$ be the approximated $N$-particle probability density at time $t = k\Delta t$. Given $\bar P_{k-1}$, then we define $\bar P_{k}$ as any solution of the variational problem 
\begin{subequations}
\label{variationalNa}
\begin{equation} \label{variationalN_1}
\underset{P_k}{\text{inf}} \underset{(P, \vec U)}{\text{inf}} \left \{ \frac{1}{2}  \int_0 ^{ \Delta t} \int_{\Omega^N} P \| \vec U \|^2 \ud \vx \ud s + \EN(P_k) \right \},
\end{equation}
where the infimum is taken among all the pairs $(P, \vec U)$ and `final position' $P_k = P(\cdot, \Delta t)$ such that $P_k \in \mathcal P_2(\Omega^N)$, $P: [0, \Delta t] \to \mathcal P_2(\Omega^N)$ with
\begin{alignat}{3}
\label{variationalN_2}
	\frac{\partial P}{\partial s} + \nabla_{\vx} \cdot (P \vec U) &= 0,  &\qquad & \text{in } \Omega^N \times (0,\Delta t),\\ 
\vec U \cdot \vec n &= 0, & & \text{on }\partial \Omega^N \times (0,\Delta t),\\
P &= \bar P_{k-1}(\vx), &&  \text{in } \Omega^N \times \{ 0\}, \\ 
P_k(\vx) & = P, &&  \text{in } \Omega^N \times \{ \Delta t \}.
\end{alignat}
The scheme is initialized for $k = 0$ with $\bar P_0 = P_0$ as given in Assumptions \ref{assum_initial}.
\end{subequations}

Starting from the $N$-particle problem \eqref{variationalNa}, we derive an analogous problem for the one-particle density and the associated flow.
In general there might be a uniqueness issue in the determination of the flow, which is related to the problem of tilting gradient flows (cf. \cite{Mielke:2020ut}). However, we have a natural convention in our case, since we can enforce consistency with an equation of non-interacting particles, that is, the variational formulation of the Fokker-Planck equation (cf. \cite{Jordan:1998wa}). Following the wide spread rationale that the Wasserstein metric is the right one for this equation (and it is also tilt invariant for changing external potentials) will fix the metric structure and entropy also at higher order as we shall see later in the proof.

\begin{assum}[Flow variable] \label{assum_flow} 
The macroscopic flow $\phi$ is defined such that, in the absence of interactions ($u = 0$), it is consistent with the fluid-dynamic formulation of the Wasserstein metric. 
\end{assum}

\begin{prop} \label{prop:N2}
Consider a short range repulsive interaction potential $u$ satisfying Assumptions \ref{assum_u}, and the problem \eqref{variationalNa} with $N=2$ and initial data satisfying Assumption \ref{assum_initial} in the asymptotic regime defined by Assumption \ref{assum_regime} with the consistency relation in Assumption \ref{assum_flow}. 
Then the one-particle marginal density $p$ of the optimality conditions of \eqref{variationalN_1}  and associated flow $ \phi$  satisfies the following equations up to order $\epsilon^d$:
\begin{alignat*}{3}
\frac{\partial p}{\partial s} + \nabla_{\x_1} \cdot (p \nabla_{\x_1} \phi) &= 0, &  \qquad & \text{in } \Omega \times (0,\Delta t),\\
\frac{\partial \phi}{\partial s} + \frac{1}{2} \| \nabla_{\x_1} \phi \| ^2 &= 0, & & \text{in } \Omega \times (0,\Delta t), \\ 
\nabla_{\x_1} \phi \cdot \bfn_1 &= 0, & & \text{on } \partial \Omega \times (0,\Delta t),\\
p &=  \bar p_{k-1}(\x_1), & & \text{in } \Omega \times \{ 0 \},\\
\phi &=  - \left( \log   \bar p_k + V +  \alpha_u   \epsilon^d   \bar p_k \right) = -\frac{\delta \Eed(\bar p_k)}{\delta p_k}, & & \text{in } \Omega \times \{ \Delta t \}.
\end{alignat*}
where $\Eed$ is given in \eqref{entropy_1}, $\alpha_u$ in \eqref{alpha_u0}, and $\bar p_{k}$ is the approximated one-particle marginal $p$ at time $t = k \Delta t$. 
Furthermore, $\bar P_k(\x_1,\x_2)$ for  $k>0$ inherits the structure of $k=0$ in Assumption \ref{assum_initial} of being chaotic for $\| \x_1 - \x_2\| \gg \epsilon$, that is the leading order term of $\bar P_k(\x_1,\x_2)$ is $\bar p_k(\x_1) \bar p_k(\x_2)$.
\end{prop}

\begin{remark}[Validity of the asymptotic expansion]
The result in Proposition \ref{prop:N2} obtains the formal asymptotic expansion of the one-particle pair $(p, \phi)$ up to order $\epsilon^d$ assuming smoothness of its terms. The validity of the asymptotic expansion, that is, showing that the rest of terms are of lower order, is an open problem. 
\end{remark}

\begin{corollary}[Hard sphere particles] \label{cor:hs}
	The result stated in Proposition \ref{prop:N2} is also valid for the hard-sphere potential $u(r/\epsilon)$ with $u (r)= + \infty$ for $r< 1$, $u(r) = 0$ otherwise. This corresponds to hard sphere particles with diameter $\epsilon$. The coefficient in the final time condition is $\alpha_u = V_d(1)$, the volume of the unit ball.
\end{corollary}

\begin{corollary}[$N$ particles case] \label{cor:Ncase}
	The result stated in Proposition \ref{prop:N2} and Corollary \ref{cor:hs} formally extend to any number of particles $N$ under Assumptions \ref{assum_regime}, that is, that the total volume of interaction is small compared to the macroscopic volume $\Omega$. The final condition reads
$$\phi(\x, \Delta t) = - \left[ \log   \bar p_k + V +  \alpha_u (N-1)   \epsilon^d  \bar p_k \right],$$
which coincides with $-\frac{\delta \Ee(\bar p_k)}{\delta p_k}$.
\end{corollary}

\begin{remark}[Macroscopic gradient-flow structure]
	The results above obtain a variational formulation and compatibility conditions for the one-particle density $p$ and the associated flow $\phi$ from the corresponding microscopic quantities satisfying \eqref{variationalNa} in the asymptotic limit $\epsilon \to 0$ given by Assumptions \ref{assum_regime}. The final step to obtain convergence to a macroscopic gradient-flow solution including the first correction term in volume fraction ($N\epsilon^d$) is to take the limit $\Delta t \to 0$ in the variational formulation. Our results show that the limits commute: taking first $\Delta t \to 0$ in \eqref{variationalNa} would take us back to the $N$-particle Fokker--Planck equation \eqref{fps}, which was the starting point in \cite{Bruna:2012cg, Bruna:2017vr} to obtain a macroscopic Fokker--Planck equation for $\epsilon \ll 0$. This equation was shown to admit a gradient-flow structure associated to the free-energy functional \eqref{entropy_1}.
	\end{remark}

\section{Variational formulation} \label{sec:variational}

In this section we define the microscopic problem in detail and present the corresponding  macroscopic problem. We then write both problems in variational form. 
For generality, we expose the problem for soft particles interacting via a repulsive potential (in subsection \ref{sec:hs} we discuss how the hard-core particles case can be seen as a particular limit of soft spheres).

\subsection{Models for soft spheres} \label{sec:softspheres}

We consider the problem satisfied by the joint law of the $N$-particle system $P(\vx,t)$, where each particle evolves according to \eqref{ssde0} with $\chi = 1$ and $\ell = \epsilon \ll 1$:
\begin{subequations}
\label{fpsb}
\begin{alignat}{3}
\label{fps_eqb}
\frac{\partial P}{\partial t} &= \nabla_{\vx} \cdot \left [  \nabla_{\vx} P + \nabla_{\vx} V_N (\vx) P + \nabla_{\vx} U_N^\epsilon (\vx) P \right ], & \qquad &\vx \in \Omega^N, t>0,\\
0 & =  \vec n \cdot \left [  \nabla_{\vx} P + \nabla_{\vx} V_N (\vx) P + \nabla_{\vx} U_N^\epsilon (\vx ) P \right ], &  &\vx \in \partial \Omega^N, t>0, \\
P & = P_0, & & \vx \in \Omega^N, t=0,
\end{alignat}
where $\vec n$ is the outward normal vector on $\partial \Omega^N$, $V_N(\vx) = \sum_{i=1}^N V(\x_i)$ is the total external potential, $U_N$ is the total interaction potential of the system given by \eqref{interact_pot} and $N$, $\epsilon$, and $P_0$ are such that  Assumptions \ref{assum_regime} and \ref{assum_initial} are satisfied. 
The corresponding  free-energy is given by
\begin{equation}
\label{entropysb}
\EN(P)  = \int_{\Omega^N}  \left[ P(\vx, t) \log P(\vx,t) + V_N (\vx) P  + U_N^\epsilon(\vx) P \right] \, \ud \vx.
\end{equation}
\end{subequations}

As discussed in the introduction, using the method of matched asymptotics on \eqref{fpsb} under Assumptions \ref{assum_regime} results in the following nonlinear Fokker--Planck equation for the one-particle marginal $\pe (\x, t)$ (valid up to order $\epsilon^d$) \cite{Bruna:2017vr}
\begin{subequations} \label{fps_one}
\begin{alignat}{3}
\label{fps_one_eq}
\frac{\partial \pe}{\partial t}  &= \nabla_\x \cdot \left \{  \left[1 + \alpha_u (N-1) \epsilon^d \pe \right] \nabla_\x \pe +  \nabla_\x  V (\x) \, \pe \right \}, & \qquad & \x \in \Omega, t>0,\\
0 &= {\bf n} \cdot \left \{  \left[1 + \alpha_u (N-1) \epsilon^d \pe \right] \nabla_\x \pe +  \nabla_\x  V (\x) \, \pe \right \}, &  & \x \in \partial \Omega, t>0,\\
\pe &= p_0^\epsilon, && \x \in \Omega, t=0,
\end{alignat}
where $p_0^\epsilon(\x) = \int_{\Omega^N} P_0(\vx) \delta(\x_1-\x)\ud \vx$.
The associated free-energy to \eqref{fps_one_eq} is 
\begin{equation}
\label{entropys_1}
\Ee ( \pe) = \int_{\Omega} \left[ \pe \log \pe + \frac{1}{2} \alpha_u (N-1) \epsilon^d (\pe)^2 + V(\x) \pe \right]  \, \ud \x.
\end{equation}
\end{subequations}

We want to study the correspondence between the microscopic model \eqref{fpsb} and the macroscopic model \eqref{fps_one} using the free energy and gradient flow description. In other words, is the gradient flow of \eqref{entropys_1} the macroscopic counterpart of the gradient flow of \eqref{entropysb} in the limit given by Assumptions \ref{assum_regime}? In order to gain some insight into this question we will consider the minimizing movement scheme, a time-discrete variational approximation of gradient flows \cite{Jordan:1998wa}. We then perform an asymptotic expansion in an outer region, where it is natural to assume the asymptotic independence of two particles, and in an inner region at the scale of particle sizes, where the size exclusion leads to significant dependence. In order to argue meaningfully with the first-order asymptotics, we also need to understand the linearized problem and its uniqueness. These building blocks will be discussed in the next subsections. 

\subsection{Minimizing movement scheme}

It is well-known that, under certain conditions of the drift, the stationary solutions of the Fokker--Planck equations \eqref{fps_eqb} and \eqref{fps_one_eq} satisfy a variational principle. Namely, they minimize their associated  free-energy functionals $\EN$ and $\Ee$ over the associated class of probability densities, $\mathcal P_2 (\Omega^N)$ and $\mathcal P_2 (\Omega)$ respectively. We will use an extra connection between the  free-energy functional and the Fokker--Planck equation, namely that the solutions of the Fokker--Planck equation follow, at each instant in time, the direction of the steepest descent of the associated free energy functional. In fact, Jordan \etal \cite{Jordan:1998wa} showed that the linear Fokker--Planck equation can be obtained as the limit of a variational scheme. Let $\bar P_{k}$ be the approximated $N$-particle probability density at time $t = k\Delta t$. Given $\bar P_{k-1}$, then we define $\bar P_{k}$ as any solution of the variational problem 
\begin{equation}
\label{discrete1}
\underset{ P_k \in \mathcal P_2(\Omega^N)}{\text{inf }} \left \{ \frac{1}{2} W_2^2 \left(P_k, \bar P_{k-1} \right) + \EN(P_k) \right \},
\end{equation}
where $W_2$ the Wasserstein metric. In order to find the Euler--Lagrange optimality conditions of \eqref{discrete1}, we make use of the Benamou--Brenier formulation of $W_2$ \cite{benamou2000computational} to rewrite \eqref{discrete1} as
\begin{subequations}
\label{constraints1}
\begin{equation} \label{discrete3}
\underset{P_k}{\text{inf}} \underset{(P, \vec U)}{\text{inf}} \left \{ \frac{1}{2}  \int_0 ^{ \Delta t} \int_{\Omega^N} P \| \vec U \|^2 \ud \vx \ud s + \EN(P_k) \right \},
\end{equation}
where the infimum is taken among all the pairs $(P, \vec U)$ and `final position' $P_k$ such that $P_k \in \mathcal P_2(\Omega^N)$, $P: [0, \Delta t] \to \mathcal P_2(\Omega^N)$ with
\begin{alignat}{3}
\label{constraint31}
	\frac{\partial P}{\partial s} + \nabla_{\vx} \cdot (P \vec U) &= 0,  &\qquad & \text{in } \Omega^N \times (0,\Delta t),\\ 
\vec U \cdot \vec n &= 0, & & \text{on }\partial \Omega^N \times (0,\Delta t),\\
P &= \bar P_{k-1}(\vx), &&  \text{in } \Omega^N \times \{ 0\}, \\ 
P &= P_k(\vx), &&  \text{in } \Omega^N \times \{ \Delta t \}.
\end{alignat}
\end{subequations}
This can be understood as an optimal control problem: we need to find the best end point $\bar P_k$ so that $\EN(\bar P_k)$ is the smallest, but also the optimal path with flux $P \vec U$ to get there. We note that in the literature is also common to find this problem defined in the time interval $[0,1]$; in this case the dependency on  $\Delta t$ appears as a coefficient $1/\Delta t$ outside the integral in \eqref{discrete3}. Benamou and Brenier pointed out that the variational scheme \eqref{constraints1} can be seen as a convex minimization problem with linear constraints. The constraint, or the continuity equation \eqref{constraint31}, will be eliminated by introducing a Lagrange multiplier $\Phi$ in the next subsection.

\subsection{Weak formulation and compatibility conditions}

Following the procedure by Brenier \cite{Brenier:2003vb}, the variational problem \eqref{constraints1} is equivalent to
\begin{equation*}
\begin{aligned}
\underset{P_k}{\text{inf}} \underset{(P, \vec U)}{\text{inf}}  \underset{\Phi }{\text{sup}} \, \bigg \{ & \int_0 ^{ \Delta t} \int_{\Omega^N} \left( \frac{\| \vec U \|^2}{2} - \partial_s \Phi - \vec U \cdot \nabla_{\vx} \Phi \right) P  \ud \vx \ud s + \int_{\Omega^N}  \Phi(\vx,\Delta t) P_k \ud \vx \\
& - \int_{\Omega^N}  \Phi(\vx,0) \bar P_{k-1} \ud \vx 
+ \EN(P_k) \bigg \},
\end{aligned}
\end{equation*}
which leads to the classical optimality conditions \cite{santambrogio2015optimal,villani2003topics}
\begin{subequations}
	\label{opticond}
\begin{alignat}{3}
\label{opticond1}
\nabla_{\vx} \Phi &= \vec U, &\qquad & \text{in } \Omega^N \times (0,\Delta t),\\
\label{opticond2}
\Phi(\vx,\Delta t) &= - \frac{ \delta \EN(\bar P_k)}{\delta P_k}, &&  \text{in } \Omega^N \times \{ \Delta t \}\\
\label{opticond3}
\frac{\partial \Phi}{\partial s} + \frac{1}{2} \| \nabla_{\vx} \Phi \| ^2 &= 0, & & \text{in } \Omega^N \times (0,\Delta t).
\end{alignat}
\end{subequations}
Here $\Phi$ can be interpreted as a Lagrange multiplier for conditions \eqref{constraints1}. 

This yields the following variational problem:
given the initial distribution $\bar P_{k-1}$ at time $t = {(k-1) \Delta t}$ and a free-energy $\EN$, determine the pair $(P, \Phi)$ and $\bar P_k(\vx) = P(\vx,\Delta t)$ such that
\begin{subequations}
\label{micro_opti}
\begin{alignat}{3} 
\label{micro_prob}
\frac{\partial P}{\partial s} + \nabla_{\vx} \cdot (P \nabla_{\vx} \Phi) &= 0, & \qquad &\text{in } \Omega^N \times (0,\Delta t),\\
\label{optim2}
\frac{\partial \Phi}{\partial s} + \frac{1}{2} \| \nabla_{\vx} \Phi \| ^2 &= 0, &  &\text{in } \Omega^N \times (0,\Delta t),\\ 
\label{optim4}
\nabla_{\vx} \Phi \cdot \vec n &= 0, & & \text{on } \partial \Omega^N \times (0,\Delta t),\\
P &= \bar P_{k-1}(\vx), & &  \text{in } \Omega^N \times \{ 0 \},\\
\label{optim3}
\Phi &= - \frac{ \delta \EN(\bar P_k)}{\delta P_k} , & & \text{in } \Omega^N \times \{ \Delta t \}.
\end{alignat}
\end{subequations}

We can repeat the procedure to obtain the weak formulation and optimality conditions for the macroscopic Fokker--Planck equation \eqref{fps_one_eq}. We arrive at
\begin{subequations}
\label{macro_opti}
\begin{alignat}{3}
\label{macro1}
\frac{\partial \pe}{\partial s} + \nabla_{\x} \cdot (\pe \nabla_{\x} \phie) &= 0, &  \qquad & \text{in } \Omega \times (0,\Delta t),\\
\label{macro2}
\frac{\partial \phie}{\partial s} + \frac{1}{2} \| \nabla_{\x} \phie \| ^2 &= 0, & & \text{in } \Omega \times (0,\Delta t), \\ 
\label{macro3}
\nabla_{\x} \phie \cdot \bfn &= 0, & & \text{on } \partial \Omega \times (0,\Delta t),\\
\label{macro4}
\pe &= \bar p_{k-1}^\epsilon(\x), & & \text{in } \Omega \times \{ 0 \}\\
\label{macro5}
\phie &= - \frac{ \delta \Ee(\bar p_k^\epsilon)}{\delta p_k^\epsilon}, & & \text{in } \Omega \times \{ \Delta t \},
\end{alignat}
where $\Ee$ is given in \eqref{entropys_1}.   deduce the convergence of the gradient-flow solutions from this

For ease of notation, from now on we will eliminate the bars on $\bar P_{k-1}$ and $\bar P_k$ from the optimality conditions, and analogously for the macroscopic densities. 
\end{subequations}

\subsection{Uniqueness of solutions of the optimality conditions}  \label{sec:uniqueness_nonlinear}

In the following we study the system of nonlinear compatibility conditions in a unified way such that it comprises the microscopic as well as the macroscopic problem. We thus consider a domain $D \subset \mathbb{R}^M$ of arbitrary dimension $M$; $D$ can be the domain in $\mathbb R^{dN}$ for the microscopic model (perforated in the case of hard spheres), or $D=\Omega$ with $M=d$ in the macroscopic model. 
We look for solutions $(p,\phi)$ of
\begin{subequations}
\label{general_opti}
\begin{alignat}{3}
\label{general1}
\frac{\partial p}{\partial s} + \nabla  \cdot (p \nabla \phi) &= 0, &  \qquad &\text{in } D \times (0,\Delta t),\\
\label{general2}
\frac{\partial \phi }{\partial s} + \frac{1}{2} \| \nabla  \phi \| ^2 &= 0, & & \text{in } D \times (0,\Delta t) \\ 
\label{general3}
p \nabla  \phi  \cdot \bfn &= 0, & & \text{on } \partial D \times (0,\Delta t),\\
\label{general4}
p &= p_0, & & \text{in } D \times \{0\}\\
\label{general5}
\phi &= - \frac{ \delta F}{\delta p}(p), & & \text{in } D \times \{\Delta t\}.
\end{alignat}
\end{subequations}
Here, $F$ is a strictly convex functional in the classical sense as \eqref{entropy_1} and \eqref{entropys}, which we assume to be differentiable for the sake of simplicity, but an analogous proof based on subgradients can be carried out in general.  

In order to verify the uniqueness of a solution we can follow the formal proof of Lasry and Lions 
\cite{lasry2007mean} for mean-field games, which have the same structure as \eqref{general_opti}. The key idea is to take the difference of the equations for two solutions $(p_i,\phi_i)$, $i=1,2$, that is, 
\begin{align*}
\frac{\partial p_1}{\partial s} - \frac{\partial p_2}{\partial s} &=  - \nabla  \cdot (p_1 \nabla \phi_1 - p_2 \nabla \phi_2), \\
\frac{\partial \phi_1 }{\partial s} -\frac{\partial \phi_2 }{\partial s} &= - \frac{1}{2} \| \nabla  \phi_1 \| ^2 + \frac{1}{2} \| \nabla  \phi_2 \| ^2 ,
\end{align*} 
then multiply the first equation with $\phi_1-\phi_2$, the second with $p_1-p_2$ and  integrate with respect to space and time. This, together with integration by parts, yields 
\begin{align*}
 - \Big \langle & \frac{ \delta F}{\delta p}(p_1(\cdot,\Delta t)) - \frac{ \delta F}{\delta p}(p_2(\cdot,\Delta t)), p_1 (\cdot,\Delta t)- p_2(\cdot,\Delta t) \Big \rangle \\ 
& = \int_D \left[\phi_1(\cdot,\Delta t) - \phi_2(\cdot,\Delta t) \right] \left[ p_1(\cdot,\Delta t) - p_2(\cdot,\Delta t) \right]\, \ud x \\
& = \int_0^{\Delta t} \! \! \int_D \frac{\partial }{\partial s} \left[(p_1-p_2)(\phi_1-\phi_2)\right] \, \ud x \ud s \\
& = - \int_0^{\Delta t} \! \! \int_D \Big [\nabla  \cdot (p_1 \nabla \phi_1 - p_2 \nabla \phi_2) (\phi_1-\phi_2) + \frac{1}2(p_1-p_2) (\| \nabla  \phi_1 \| ^2 - \| \nabla  \phi_2 \| ^2 )
\Big]\, \ud x \ud s \\
& =  \frac{1}2 \int_0^{\Delta t} \int_D (p_1+p_2) \| \nabla  \phi_1 - \nabla \phi_2 \| ^2 \, \ud x \ud s. 
\end{align*} 
For nonnegative $p_1$ and $p_2$, the right-hand side is obviously nonnegative, while the left-hand side is negative due to the strict convexity of $F$ if $p_1(\cdot,\Delta t) \neq p_2(\cdot,\Delta t)$. Thus we conclude $p_1(\cdot,\Delta t) = p_2(\cdot,\Delta t)$ and $\nabla  \phi_1 = \nabla \phi_2$ on the support of $p_1+p_2$. The uniqueness of the transport equation \eqref{general1} thus implies $p_1 \equiv p_2$.

\subsection{The linearized compatibility conditions}

In order to justify  an asymptotic expansion it is a key issue to understand  the linearized problem and its well-posedness. In the following we will provide an analysis based on the Ladyzhenskaya--Babuska--Brezzi theory  \cite{Brezzi:1991tn} for linear saddle-point problems under suitable conditions, which can be carried out in a dimension-independent way, hence being applicable for the high-dimensional microscopic as well as for the macroscopic problem. As before, we consider a general domain $D \subset \mathbb{R}^M$. 
 
Given a known pair $(q, \varphi)$, around which we linearize the optimality conditions \eqref{general_opti}, we obtain the following linear problem for $(h, f)$:
\begin{subequations}
\label{conditions_hf}
\begin{alignat}{3}
\label{con3}
0 & = \frac{\partial h}{\partial s} + \nabla \cdot (q \nabla f  + h \nabla \varphi ) \qquad & & \text{in } D \times (0,\Delta t),\\
\label{con1}
0 & = \frac{\partial f}{\partial s} + \nabla \varphi  \cdot  \nabla f & & \text{in } D \times (0,\Delta t),\\
0 &= (q \nabla f  + h \nabla \varphi )\cdot n, & & \text{on } \partial D \times (0,\Delta t),\\  
\label{con4}
h  & = h_0 & & \text{in } D \times \{0\}, \\
\label{con2}
f   &=  - \frac{h}{q}  - C q & & \text{in } D \times \{\Delta t\},
\end{alignat}
\end{subequations}
for a nonnegative constant $C$.

Here we prove the existence and uniqueness of the linearized system \eqref{conditions_hf}. 
In order to verify existence and uniqueness of a weak solution of  \eqref{conditions_hf} we introduce a weak formulation, which actually corresponds to a second order approximation of the original variational problem. We will assume that $q$ and $\varphi$ are sufficiently smooth and $q$ is strictly positive in $D \times [0,\Delta t]$. Consequently we introduce a variable corresponding to the flux, that is, a vector field $g=q \nabla f$. Then consider the variational problem of minimizing
\begin{equation} \label{eq:linearsaddle1}
\inf_{g,h} \left( \frac{1}2 \int_0^{\Delta t} \int_D \frac{|g|^2}{q} \, \ud x \ud s + \frac{1}2 \int_D \frac{h(\cdot,\Delta t)^2}{q} \, \ud x  + C \int_D h(\cdot,\Delta t){q(\cdot,\Delta t)} \, \ud x \right),
\end{equation}
subject to 
\begin{equation} \label{eq:linearsaddle2}
 \frac{\partial h}{\partial s} + \nabla  \cdot (g + h \nabla \varphi ) = 0.
\end{equation}
Introducing $f$ as a Lagrange parameter, we obtain the saddle-point problem
\begin{align*}
	\inf_{g,h} \sup_f &\left\{ \frac{1}2 \int_0^{\Delta t} \int_D \frac{|g|^2}{q} \, \ud x \ud s + \frac{1}2 \int_D \frac{h(\cdot,\Delta t)^2}{q(\cdot,\Delta t)} \, \ud x  + C \int_D h(\cdot,\Delta t){q(\cdot,\Delta t)} \, \ud x\right. 
	\\ & \qquad\qquad\qquad\qquad\qquad\left. + \int_0^{\Delta t} \int_D ( \frac{\partial h}{\partial s} f - (g + h \nabla  \varphi )\cdot \nabla  f ) \, \ud x \ud s \right\}. 
\end{align*}
Clearly an appropriate space for $g$ is $L^2((0,\Delta t) \times D)^M$ due to the first term in the functional. Since we expect $\nabla f = \frac{g}q$, an obvious choice is $f  \in L^2(0,\Delta t;H^1(D))$. Hence, the constraint on $h$ and $g$ is to be interpreted in the dual space $L^2(0,\Delta t;H^1(D)^*)$, which makes sense also since $ \nabla  \cdot g \in L^2(0,\Delta t;H^1(D)^*)$.

It remains to define an appropriate space for $h$, which will be based on the method of characteristics. First of all, let $r$ denote the push-forward of $h(\cdot,0)$, that is, the unique solution of 
\begin{equation*}
	\frac{\partial r}{\partial s} + \nabla  \cdot (r \nabla  \varphi ) = 0,
\end{equation*}
with initial value $h(\cdot,0)$. Then we look for a distributional solution $f \in r + W$, where 
\begin{equation*}
	W= \left \{h \in  L^2(0,\Delta t;H^1(D)^*)~ \Big |~\frac{\partial h}{\partial s} + \nabla  \cdot (  h \nabla  \varphi ) \in L^2(0,\Delta t;H^1(D)^*), h(\cdot,0)=0. \right \}.
\end{equation*}
Note that, for $\frac{\partial h}{\partial s} + \nabla  \cdot (  h \nabla  \varphi ) \in L^2(0,\Delta t;H^1(D)^*)$ given, one can reconstruct $h$ with zero initial value by appropriate integration along characteristics and obtains also a distributional trace at $s=\Delta t$. 

In the following we thus look for a solution
\begin{equation*}
	(f,g,h) \in  L^2(0,\Delta t;H^1(D))
 \times L^2((0,T)\times D)^M \times (r+W),
\end{equation*}
using the general theory of saddle-point problems in Hilbert spaces \cite{Brezzi:1991tn}.
 
 \begin{theorem} \label{theorem:linear_var}
 Let $q \in C(0,\Delta t;L^\infty(D))$ be positive, $r \in  L^2(0,\Delta t;H^1(D)^*)$, and $\varphi \in C(0,\Delta t;W^{1,\infty}(D))$.
 Then there exists a unique solution 
 \begin{equation*}
 	(f,g,h) \in  L^2(0,\Delta t;H^1(D))
 \times L^2((0,T)\times D)^M \times (r+W),
 \end{equation*}
 of the variational problem \eqref{eq:linearsaddle1} subject to \eqref{eq:linearsaddle2}, respectively a weak solution of \eqref{conditions_hf}.
 \end{theorem}
\begin{proof}
Following \cite[Theorem 1.1, p. 42]{Brezzi:1991tn}, we need to verify an inf-sup condition for the constraint and the coercivity of the quadratic functional on the kernel of the constraint, that is, in the setting of \cite{Brezzi:1991tn}
\begin{equation*}
	a(g,h;\tilde g,\tilde h) = \int_0^{\Delta t} \int_D \frac{g~\tilde g}{q} \, \ud x \ud s + \frac{1}2 \int_D \frac{h(\cdot,\Delta t) \tilde h(\cdot,\Delta t)}{q(\cdot,\Delta t)} \, \ud x
\end{equation*}
and
\begin{equation*}
	b(f;g,h) = \int_0^{\Delta t} \int_D ( \frac{\partial h}{\partial s} f - (g + h \nabla  \varphi )\cdot \nabla f ) \, \ud x \ud s.
\end{equation*}
The inf-sup condition follows immediately by estimating the supremum over all $g,h$ by the value at  $g=-\lambda \nabla  f$ with $\lambda> 0$ sufficiently large, and $h$ constant in space defined via 
$\partial_s h(s) = \int_D f(x,s) \, \ud x $, that is,
\begin{align*} 
\sup_{g,h} \frac{\int_0^{\Delta t} \int_D ( \frac{\partial h}{\partial s} f - (g + h \nabla  \varphi )\cdot \nabla f ) \, \ud x \ud s}{\Vert(g,h) \Vert}  &\geq  c_1 \frac{
\int_0^{\Delta t} (\int_D f \, \ud x)^2 \, \ud s + \int_0^{\Delta t} \int_\Omega |\nabla  f|^2 \, \ud x \ud s}{\Vert (h,- \lambda \nabla  f) \Vert} \\
 &\geq c_2  \sqrt{\int_0^{\Delta t} \left(\int_D f\, \ud x \right)^2 \, \ud s + \Vert \nabla f \Vert_{L^2}^2}. 
\end{align*}
Finally the Poincar\'e inequality implies that the right-hand side can be estimated from below by a multiple of the norm of $f$ in $L^2(0,T;H^1(D))$.

For coercivity, we restrict ourselves to $\frac{1}2 \int_0^{\Delta t} \int_\Omega \frac{g^2}{q} \, \ud x \ud s$, which is clearly coercive with respect to $g$ in $L^2((0,\Delta t) \times D)$. However, for $g \in L^2((0,\Delta t) \times \Omega)$ and $(g,h)$ in the kernel of the constraint we immediately have
\begin{align*} \Vert h \Vert_W^2 &= \int_0^{\Delta t} \left \Vert \frac{\partial h}{\partial s} + \nabla  \cdot (  h \nabla \varphi ) \right \Vert_{ L^2(0,\Delta t;H^1(D)^*)}^2 \, \ud s =
\int_0^{\Delta t} \Vert   \nabla  \cdot g \Vert_{ L^2(0,\Delta t;H^{-1}(D))}^2 \, \ud s 
\\ &\leq \Vert g \Vert_{L^2((0,\Delta t) \times D)}^2,
\end{align*}
which implies also coercivity.
\end{proof}

\section{Derivation of the macroscopic variational Fokker--Planck equation} \label{sec:onespecies}

In this section, we show that the macroscopic compatibility conditions \eqref{macro_opti} can be derived from the corresponding microscopic problem \eqref{micro_opti}. We begin by the simple case of noninteracting particles, and then consider the case of interacting particles. We show the derivation for soft spheres, and conclude the section presenting the key differences in deriving the macroscopic variational problem for hard spheres.

In order to reduce the dimensionality of the problem \eqref{micro_opti}, we consider the marginal densities
\begin{align} \label{marginalPn}
	P_n(\x_1,\dots, \x_n, s) &= \int_{\Omega^{N-n}} P (\vx, s) \ud \x_{n+1} \dots \x_N,
\end{align}
for $n = 1, 2, \dots N-1$.
Integrating \eqref{micro_prob} over $\ud \x_2 \dots \ud \x_N$ using the boundary conditions \eqref{optim3} gives
\begin{align}
\label{integ1}
\frac{\partial p}{\partial s}(\x_1, s) + \nabla_{\x_1} \cdot  \int_{\Omega^{N-1}} \left( P \nabla_{\x_1} \Phi \right) \, \ud \x_2 \dots \ud \x_N = 0,
\end{align}
where $p$ is the one-particle marginal density ($p \equiv P_1$),
\begin{equation}
\label{marginal-s}
p(\x_1, s) = \int_{\Omega} P_2(\x_1, \x_2, s) \, \ud \x_2.
\end{equation}

\subsection{Non-interacting particles} \label{sec:points}

We begin by the simplest case of non-interacting particles, so that the interaction potential is $u \equiv 0$. Using that particles are initially independent and identically distributed, we can write 
\begin{equation} \label{point_indep}
	P(\vx,s) = \prod_{i=1}^N  p(\x_i,s).
\end{equation}
 
Using \eqref{point_indep} and \eqref{optim2}-\eqref{optim3}, the problem for $\Phi$ reads
\begin{subequations}
\label{optimpoint}
\begin{align}
\label{optim2b}
0 &= \frac{\partial \Phi}{\partial s} + \frac{1}{2} \left(  \| \nabla_{\x_1} \Phi \| ^2 + \cdots +  \| \nabla_{\x_N} \Phi \| ^2 \right),\\
\label{final_condpoint}
\Phi(\vx,\Delta t) &= - \sum_{i=1}^N \left[ \log  p_{k}(\x_i) + V(\x_i) \right],
\end{align}
\end{subequations}
Using the decomposition of the flux
\begin{equation}
\label{phi_sep}
\Phi(\vx,s) = \sum_{i=1}^N \varphi(\x_i,s),
\end{equation}
we find that $\varphi(\x, s)$ satisfies
\begin{subequations}
\label{varphimodel}
\begin{align}
\label{varphieq}
0 & = \frac{\partial \varphi}{\partial s}  + \frac{1}{2}  \| \nabla_{\x} \varphi  \| ^2 ,\\
\label{varphicond} 
\varphi(\x,\Delta t)  &= -\log  p_{k} (\x) - V(\x). 
\end{align}
\end{subequations}
This solution is unique (Section \ref{sec:uniqueness_nonlinear}). 
Finally, inserting \eqref{phi_sep} into \eqref{integ1} gives
\begin{align}
\label{integ2}
\frac{\partial p}{\partial s}(\x_1,s) + \nabla_{\x_1} \cdot   \left( p \nabla_{\x_1} \varphi \right) = 0.
\end{align}
Therefore, we have obtained all the macroscopic compatibility conditions \eqref{macro_opti} as required (with $u= 0$). We have shown that, if $(p, \varphi)$  verify \eqref{varphimodel} and \eqref{integ2}, then  $(P, \Phi)$ given by \eqref{point_indep} and \eqref{phi_sep} satisfy \eqref{micro_opti} with $u=0$.

\subsection{Soft-sphere particles: case \texorpdfstring{$N=2$}{N=2}} \label{sec:soft_spheres_limit_N2}

For $N=2$, consider a decomposition of the two-particles flow of the form \cite{Burger:2017ob}
\begin{equation}\label{phi_truncated2}
	\Phi_2(\x_1, \x_2, s) =  \varphi(\x_1,s) + \varphi(\x_2,s) + \varphi_2(\x_1,\x_2, s).
\end{equation}
Then equation \eqref{integ1} reads
\begin{align}
\label{integ1b_N2}
\frac{\partial p}{\partial s} + \nabla_{\x_1} \cdot  \left[ p \nabla_{\x_1} \varphi + \int_{\Omega} P_2(\x_1,\x_2,s) \nabla_{\x_1} \varphi_2(\x_1, \x_2,s)  \, \ud \x_2 \right],
\end{align}

The problem \eqref{micro_opti} for $(P_2, \Phi_2)$ reads
\begin{subequations}
\label{micro_optis}
\begin{alignat}{3}
\label{micro_probs}
0 & = \frac{\partial P_2}{\partial s} + \nabla_{\x_1} \cdot \left(P_2 \nabla_{\x_1}  \Phi_2 \right) + \nabla_{\x_2} \cdot \left(P_2 \nabla_{\x_2}  \Phi_2 \right), \qquad & & \text{in } \Omega^2 \times (0,\Delta t), \\
\label{optim2-s}
0 &=  \frac{\partial \Phi_2}{\partial s} + \frac{1}{2} \| \nabla_{\x_1 } \Phi_2 \| ^2  + \frac{1}{2} \| \nabla_{\x_2 } \Phi_2 \| ^2,   & & \text{in } \Omega^2 \times (0,\Delta t), \\ 
\label{optim4-s}
0 & = P_2 \nabla_{\x_i}  \Phi_2 \cdot {\bf n} ,   & & \text{on } \partial\Omega^2 \times (0,\Delta t),\\
P_2 &= P_{2, k-1}( \x_1, \x_2), & & \text{in } \Omega^2 \times \{0 \}, \\
\label{optim3-s}
\Phi_2 & = -\log P_{2,k} - V(\x_1) - V(\x_2) - u( (\x_1 - \x_2)/\epsilon), & & \text{in } \Omega^2 \times \{ \Delta t \},
\end{alignat}
\end{subequations}
where $P_{n,k}$ denotes the $n$th marginal of $P_k$. 

\subsubsection{Matched asymptotic expansions} \label{sec:mae_soft}

We seek a solution to \eqref{micro_optis} using the method of matched asymptotic expansions \cite{vanDyke:1964vg}. Suppose that  when two particles are far apart ($ \| \x_1- \x_2\| \gg \epsilon$), their Brownian motions are independent, whereas when they are close to each other ($\| \x_1- \x_2\| \sim \epsilon$) they are correlated due to interactions. We designate these two regions of the configuration space $\Omega^2$ the outer region and inner region, respectively. The solution pairs in the outer and inner regions are denoted by $(\Pout, \Phi_\text{out})$ and $(\tilde P, \tilde \Phi)$ respectively. We look for a solution in each region in powers of $\epsilon$,
\begin{alignat*}{2}
    \Pout &= \Pout^{(0)} + \epsilon \Pout^{(1)} + \epsilon^2 \Pout^{(2)} + \cdots, &\qquad  \Phi_\text{out} &= \Phi_\text{out}^{(0)} + \epsilon \Phi_\text{out}^{(1)} + \epsilon^2 \Phi_\text{out}^{(2)} + \cdots,\\
    \tilde P &= \tilde P^{(0)} + \epsilon \tilde P^{(1)} + \epsilon^2 \tilde P^{(2)} + \cdots, & \tilde \Phi &= \tilde \Phi^{(0)} + \epsilon \tilde \Phi^{(1)} + \epsilon^2 \tilde \Phi^{(2)} + \cdots.  
\end{alignat*}
Now we can be more precise about the Assumptions \ref{assum_initial} for the initial data in Section \ref{sec:outline}. 
\begin{assum}[Refined Assumptions \ref{assum_initial}] \label{assum_initial_refined}  We require the outer expansion of the initial data to be of the form:
\begin{align}
\begin{aligned} \label{in_outer}
	\Pout^{(0)}(\x_1, \x_2, 0) &= q_0 (\x_1) q_0(\x_2),\qquad
    \Pout^{(l)}(\x_1, \x_2, 0)  = 0, \quad l = 1, \dots d.
    \end{aligned}
\end{align}
\end{assum}

\paragraph{Outer region} \label{sec:outerb}

In the outer region, $\| \x_1- \x_2\| \gg \epsilon$ and hence the interaction term in \eqref{optim3-s} will be small. Specifically, given the decay of $u$ at infinity in Assumptions \ref{assum_u}, the outer problem up to $O(\epsilon^d)$ does not see the interaction term.\footnote{For a particular interaction potential it could be the interaction term comes at an even higher order. For example, with a Lennard--Jones potential $u \sim r^{-6}$ as $r\to \infty$ so the outer problem up to $O (\epsilon^5)$ is interaction-free. However, for an exponential potential $ u(r) = e^{-r}$,  $u = o(r^{-n})$ for any $n$ as $r \to \infty$ (so the outer will not see the interaction at any order).} Therefore, in the outer region \eqref{micro_optis} becomes, up to $O(\epsilon^d)$
\begin{subequations}
\label{micro_outer}
\begin{alignat}{3}
\label{micro_probo}
0 & = \frac{\partial \Pout}{\partial s} + \nabla_{\x_1} \cdot \left(\Pout \nabla_{\x_1}  \Phi_\text{out} \right) + \nabla_{\x_2} \cdot \left(\Pout \nabla_{\x_2}  \Phi_\text{out} \right), & \quad & \Omega^2 \times (0,\Delta t), \\
\label{optim2-o}
0 &=  \frac{\partial \Phi_\text{out}}{\partial s} + \frac{1}{2} \| \nabla_{\x_1 } \Phi_\text{out} \| ^2  + \frac{1}{2} \| \nabla_{\x_2 } \Phi_\text{out} \| ^2,   & & \Omega^2 \times (0,\Delta t), \\ 
\label{optim4-o}
0 & = \Pout \nabla_{\x_i}  \Phi_\text{out} \cdot {\bf n} ,   & & \partial\Omega^2 \times (0,\Delta t),\\
\label{optim5-o}
\Pout &= P_{\text{out}, k-1}( \x_1, \x_2), & & \Omega^2 \times \{0 \}, \\
\label{optim3-o}
\Phi_\text{out} & = -\log P_{\text{out},k} - V(\x_1) - V(\x_2), & & \Omega^2 \times \{ \Delta t \}.
\end{alignat}
\end{subequations}

At leading order, using \eqref{in_outer}, problem \eqref{micro_outer} for $k = 1$ admits a separable solution, and  hence the same is true for all $k>1$. Therefore, we have that 
\begin{equation}
\label{outer_ansatz}
\Pout^{(0)} (\x_1, \x_2,s) = q(\x_1,s) q(\x_2,s),\qquad 
\Phi_\text{out}^{(0)} (\x_1, \x_2,s) = \varphi(\x_1,s) +  \varphi(\x_2,s),
\end{equation}
for some functions $q$ and $\varphi$. We see that, at leading order, the outer solution has the form we found for non-interacting particles in Section \ref{sec:points}, namely that the  density is a product of densities in $\x_i$ and the flow is a sum of flows in $\x_i$.
In particular, this implies that the outer density $q$  and the outer flow $\varphi$ satisfy the equations for independent particles \eqref{integ2} and \eqref{varphieq}, respectively, found in the previous subsection.  That is, 
\begin{subequations}
\label{micros_lead}
\begin{alignat}{3}
\label{micros_leadP}
0& = \frac{\partial q}{\partial s} + \nabla_{ \x} (q  \nabla_{ \x} \varphi ), \qquad & & \text{in } \Omega \times (0,\Delta t), \\
\label{micros_leadPhi}
0& = \frac{\partial \varphi}{\partial  s}  + \frac{1}{2} \| \nabla_{\x} \varphi \| ^2, & & \text{in } \Omega \times (0,\Delta t), \\
0& = q \nabla_\x \varphi \cdot {\bf n}, & & \text{on } \partial \Omega \times (0,\Delta t),\\
q &= q_{k-1} (\bf x), & & \text{in } \Omega \times \{0 \}, \\
\label{gs_fsb}
\varphi & = - \left[  \log  q_k ( \x)   + V(\x) \right ], \qquad & & \text{in } \Omega \times \{ \Delta t \},
\end{alignat}
\end{subequations}

The $O(\epsilon)$ of \eqref{micro_outer} is, using \eqref{outer_ansatz},
\begin{subequations}
\label{outer_general}
\begin{alignat}{3}
0 & = \frac{\partial \Pout^{(1)}}{\partial s} + \nabla_{\vx} \cdot \left(\Pout^{(1)} \nabla_{ \vx}  \Phi_\text{out}^{(0)} + q(\x_1) q (\x_2) \nabla_{\vx}  \Phi_\text{out}^{(1)} \right) , \quad & & \text{in } \Omega^2 \times (0,\Delta t), \\
0 &=  \frac{\partial  \Phi_\text{out}^{(1)} }{\partial s} +  \nabla_{\vx }  \Phi_\text{out}^{(0)}  \cdot   \nabla_{\vx }  \Phi_\text{out}^{(1)} ,   & & \text{in } \Omega^2 \times (0,\Delta t)\\ 
0 & = \Pout^{(1)}  \nabla_{\x_i}   \varphi(s, \x_i) \cdot {\bf n}  + q(\x_1) q (\x_2) \nabla_{ \x_i }  \Phi_\text{out}^{(1)}  \cdot {\bf n},  & & \text{on } \partial \Omega^2 \times (0,\Delta t),\\
\Pout^{(1)}  &= P_{\text{out}, k-1}^{(1)}(  \vx), & & \text{in } \Omega^2 \times \{0 \}, \\
\Phi_\text{out}^{(1)}  &= - \frac{ P_{\text{out}, k}^{(1)}( \vx) }{ q_k (\x_1) q_k (\x_2)}, & & \text{in } \Omega^2 \times \{ \Delta t \},
\end{alignat}
\end{subequations}
where $\vx = (\x_1, \x_2)$. 
Assumptions \ref{assum_initial} means that $P_\text{out,0}^{(l)} (\vx) = 0$ for all $l\ge 1$. Given the scheme to obtain the iterate $P_{\text{out},k}^{(1)}$ for $k \ge 1$, we find that the zero initial condition propagates and that $(P_{\text{out}}^{(1)}, \Phi_{\text{out}}^{(1)}) = (0,0)$ solves the problem \eqref{outer_general}. This solution is unique using Theorem \ref{theorem:linear_var}. It is straightforward to see that the same is true for the higher-order terms $(P_{\text{out}}^{(l)}, \Phi_{\text{out}}^{(l)})$ up to $l = d$. Therefore, we have found that, up to $O(\epsilon^d)$,
\begin{equation}
	\label{outer-sol}
	\Pout (\x_1, \x_2,s) = q(\x_1,s) q(\x_2,s),\qquad 
\Phi_\text{out} (\x_1, \x_2,s) = \varphi(\x_1,s) +  \varphi(\x_2,s).
\end{equation}

\paragraph{Inner region} \label{sec:inner_soft}
In the inner region, we set $\x_1 = \tilde \x_1$, $\x_2 =  \tilde \x_1 + \epsilon \tilde \x$, and define $\tilde P(\tilde \x_1 , \tilde \x,s) = P_2(\x_1 ,  \x_2,s) $ and $\tilde \Phi (\tilde \x_1 , \tilde \x,s) = \Phi_2(\x_1 ,  \x_2,s) $. With this rescaling, \eqref{micro_optis} becomes
\begin{subequations}
\label{micros_inner-s}
\begin{align}
\label{micros_innerc}
0& = \epsilon^2 \frac{\partial \tilde P}{\partial s} +   \nabla_{\tilde \x_1} \cdot (\epsilon^2\tilde P \nabla_{\tilde \x_1} \tilde  \Phi - \epsilon \tilde P \nabla_{\tilde \x} \tilde  \Phi )  +   \nabla_{\tilde \x} \cdot (2 \tilde P \nabla_{\tilde \x} \tilde  \Phi   - \epsilon \tilde P \nabla_{\tilde \x_1} \tilde  \Phi ) ,   \\
\label{micros_flow}
0& = \epsilon^2 \frac{\partial \tilde \Phi}{\partial s} + \frac{\epsilon^2 }{2} \| \nabla_{\tilde \x_1 } \tilde \Phi \| ^2 -  \epsilon   \nabla_{\tilde \x_1 } \tilde \Phi \cdot \nabla_{\tilde \x } \tilde \Phi + \| \nabla_{\tilde \x } \tilde \Phi \| ^2,
\\
\label{initPs}
\tilde P(s=0) &= \tilde P_{k-1}( \tilde \x_1, \tilde \x), \\
\label{finalphi_s}
\tilde \Phi (s=\Delta t) &=  - \left[ \log \tilde P_k( \tilde  \x_1, \tilde \x) + V(\tilde \x_1) + V(\tilde \x_1 + \epsilon \tilde \x) + u(\tilde \x) \right], 
\end{align}
where $\tilde P_{k-1}( \tilde \x_1, \tilde \x)$ is the inner expansion of $P_{k-1} (\x_1, \x_2)$. To this we need to add matching conditions on $\tilde P$ and $\tilde \Phi$ so that they tend to the outer solution as $\| \tilde \x \| \to \infty$. Expanding the outer solution  \eqref{outer-sol} in inner variables gives
\begin{align} \label{P_match}
\tilde P(\tilde \x_1, \tilde \x,s) &\sim q^2(\tilde \x_1,s) + \epsilon q(\tilde \x_1,s)  \tilde \x \cdot \nabla_{\tilde \x_1} q (\tilde \x_1,s) + \cdots, \\
\label{Phi_match}
\tilde \Phi(\tilde \x_1, \tilde \x,s) &\sim 2 \varphi(\tilde \x_1,s) + \epsilon  \tilde \x \cdot \nabla_{ \tilde \x_1} \varphi(\tilde { \bf x}_1,s) + \cdots,
\end{align}
as $\| \tilde \x \| \to \infty$.
\end{subequations}

We look for a solution of \eqref{micros_inner-s} of the form $\tilde P \sim \tilde P^{(0)} + \epsilon \tilde P^{(1)} + \cdots$  and $\tilde 
\Phi \sim \tilde \Phi ^{(0)} + \epsilon \tilde \Phi ^{(1)} + \cdots$.  The leading-order inner problem is
\begin{subequations}
\label{inner_0s}
\begin{alignat}{2}
\label{inner_0s_1}
 0 &= 2\nabla_{\tilde \x} \cdot \left (\tilde P^{(0)} \nabla_{\tilde \x} \tilde  \Phi ^{(0)} \right ) ,  & \\ 
 \label{inner_0s_2}
 0 &= \left \| \nabla_{\tilde \x } \tilde \Phi ^{(0)}  \right \| ^2,&\\
  \label{inner_0s_3}
\tilde \Phi ^{(0)} &=  - \left[ \log \tilde P^{(0)}_k( \tilde  \x_1, \tilde \x) + 2V(\tilde \x_1)   + u(\tilde \x) \right], & \qquad \text{at } s = \Delta t, \\
\label{inner_0s_4}
\tilde P^{(0)} &\sim q^2(\tilde \x_1,s), & \text{as }  \| \tilde \x \| \to \infty, \\
\label{inner_0s_5}
\tilde \Phi^{(0)} &\sim 2 \varphi(\tilde \x_1,s), & \text{as }  \| \tilde \x \| \to \infty.
\end{alignat}
\end{subequations}
Equation \eqref{inner_0s_2} implies that $\nabla_{\tilde \x } \tilde \Phi ^{(0)} = 0$, which also satisfies \eqref{inner_0s_1}. If we assign to $\tilde \Phi^{(0)}$ the value required at infinity, 
\begin{equation}
\label{sol0_ins}
\tilde \Phi^{(0)} (\tilde \x_1,s)= 2 \varphi(\tilde \x_1,s),
\end{equation}
where $\varphi$ satisfies  \eqref{gs_fsb},  we have that
\[
\tilde \Phi^{(0)} (\tilde \x_1 ,\Delta t) = 2 \varphi(\tilde \x_1,\Delta t) = - 2 \left[  \log  q_k( \tilde \x_1)   + V(\tilde \x_1) \right ].
\]
Comparing this with \eqref{inner_0s_3} we arrive at
\begin{equation*}
 \log \tilde P^{(0)}_k ( \tilde  \x_1, \tilde \x)    =     \log  q^2_k ( \tilde \x_1)  - u(\tilde \x),
\end{equation*}
which yields
\begin{equation}
\label{inner-o1_s}
\tilde P_k ^{(0)} = e^{-u(\tilde \x) } q_k^2(\tilde \x_1).
\end{equation} 
Note that this satisfies the matching condition at infinity \eqref{inner_0s_4} since the potential decays at infinity, $\lim_{r\to \infty} u(r) = 0$ (see Assumptions \ref{assum_u}). 

The $O(\epsilon)$ of \eqref{micros_inner-s} is, using \eqref{sol0_ins} and \eqref{inner-o1_s}, 
\begin{subequations}
\label{inner_1-s}
\begin{alignat}{2}
\label{inner_1-s_eq}
 0 & = \nabla_{\tilde \x} \cdot \left [  \tilde P^{(0)} \left(    \nabla_{\tilde \x} \tilde  \Phi^{(1)} -   \nabla_{\tilde \x_1} \varphi (\tilde \x_1, s) \right) \right] , &  \\
\label{final_o1-s} 
\tilde \Phi ^{(1)} &=  -    \frac{ \tilde P^{(1)} _k( \tilde \x_1, \tilde \x) }{e^{-u (\tilde \x) } q_k^2(\tilde \x_1)} -
\tilde \x \cdot \nabla_{\tilde \x_1} V(\tilde \x_1), &\qquad \text{at } s = \Delta t,\\
\label{outer_P01-s}
\tilde P ^{(1)}  &\sim  q(\tilde \x_1,s)  \tilde \x \cdot \nabla_{\tilde \x_1} q(\tilde \x_1,s), & \text{as }  \| \tilde \x \| \to \infty, \\
\label{outer_phi01-s}
\tilde \Phi ^{(1)}  &\sim   \tilde \x \cdot \nabla_{ \tilde \x_1} \varphi( \tilde { \bf x}_1,s), & \text{as }  \| \tilde \x \| \to \infty.
\end{alignat}
\end{subequations}
Note that the $O(\epsilon)$ of \eqref{micros_flow} is automatically satisfied, so disappears in \eqref{inner_1-s}. 
Solving \eqref{inner_1-s_eq} together with the matching condition \eqref{outer_phi01-s}  we find 
\begin{equation} \label{phi_1}
 \tilde  \Phi^{(1)} (\tilde \x_1, \tilde \x,s) =    \tilde \x \cdot \nabla_{\tilde \x_1} \varphi  (\tilde \x_1,s). 
\end{equation}
Combining \eqref{gs_fsb} and  \eqref{phi_1} into \eqref{final_o1-s}, we find that
\[
\tilde P^{(1)} _k (\tilde \x_1, \tilde \x)   =  e^{-u( \tilde \x)}q_k(\tilde \x_1)  (\tilde \x \cdot \nabla_{\tilde \x_1} q_k(\tilde \x_1) ).
\]

In sum, we find that the inner region solution is, to $O(\epsilon)$,
\begin{subequations}
\label{innersol}
\begin{align}
\label{Pinnersolk}
\tilde P_k ( \tilde \x_1, \tilde \x) &\sim  e^{-u (\tilde \x)} \left[ q_k^2(\tilde \x_1) + \epsilon q_k(\tilde \x_1)  \tilde \x \cdot \nabla_{\tilde \x_1} q_k(\tilde \x_1)  \right],\\
\label{Phiinnersol}
\tilde \Phi  (\tilde \x_1, \tilde \x,s) &\sim 2 \varphi(\tilde \x_1,s) + \epsilon   \tilde \x \cdot \nabla_{\tilde \x_1} \varphi  (\tilde \x_1,s),
\end{align}
where $q_k$ and $\varphi$ satisfy the outer problem \eqref{micros_lead}. We note that the inner region equations  determine completely the flow up to $O(\epsilon)$, but that we only obtain conditions for the density at the final time ($s = \Delta t$) and at infinity ($\|\tilde \x \| \sim \infty$). But since the problem is stationary up to $O(\epsilon)$ ($s$ appears only as a parameter in \eqref{inner_0s} and \eqref{inner_1-s}), we can replace $q_k  ( \tilde \x_1)$ by $q (s,  \tilde \x_1)$ in \eqref{Pinnersolk} and write
\begin{align*}
\tilde P (\tilde \x_1, \tilde \x,s) &=  e^{-u (\tilde \x)} \left[ q^2(\tilde \x_1,s) + \epsilon q(\tilde \x_1,s)  \tilde \x \cdot \nabla_{\tilde \x_1} q(\tilde \x_1,s) \right] + B_0(\tilde \x_1, \tilde \x,s) + \epsilon B_1(\tilde \x_1, \tilde \x,s),
\end{align*}
to $O(\epsilon)$, for any functions $B_i$ ($i=0,1$) such that $B_i = 0$ at $s = \Delta t$ and as $\|\tilde \x \| \sim \infty$, so we do not obtain a unique solution for the density in the inner region. However, the subsequent analysis shows that the value of the integral in \eqref{integ1b_N2} and the integrated compatibility conditions are invariant to $B_i$ and thus it what follows we can simply set them to zero:
\begin{align}
\label{Pinnersol}
\tilde P (\tilde \x_1, \tilde \x,s) &=  e^{-u (\tilde \x)} \left[ q^2(\tilde \x_1,s) + \epsilon q(\tilde \x_1,s)  \tilde \x \cdot \nabla_{\tilde \x_1} q(\tilde \x_1,s) \right] + O(\epsilon^2).
\end{align}
\end{subequations}

\subsubsection{Integrated equations} \label{sec:soft-integral}

We now go back to \eqref{integ1b_N2} and use the inner and outer solutions in order to obtain the optimization conditions for the macroscopic problem \eqref{fps_one}. First, we need the following result. 

\begin{lemma}[Relationship between $p$ and $q$] \label{lem:rel_pq}
The one-particle density $p$ and the outer density $q$ are related by
\begin{align} \label{relation_p_vq}
	p(\x_1,s) = q(\x_1,s)  \left [ \int_\Omega q(\x,s)  \, \ud \x  - \alpha_u   \epsilon^d q(\x_1,s)  \right] + o(\epsilon^d),
\end{align}
where 
\begin{equation} 
\label{alphaV}
\alpha_u =  \int_{\mathbb R^d}  \left ( 1 - e^{- u(\epsilon  \x )} \right)  \ud  \x,
\end{equation}
and $\int_\Omega q(\x,s)  \, \ud \x = 1 +  \epsilon^d a + o(\epsilon^d)$, where $a$ is an order one constant. Therefore, $p = q + O(\epsilon^d)$. 
\end{lemma}

\begin{proof}
To keep the notation simple, in the following we omit the time variable $s$ as an argument of the densities. 
We begin by evaluating the integral in \eqref{marginal-s} by splitting the integration volume $\Omega$ for $\x_2$ into the inner and outer regions and using the inner and outer solutions for $P_2(\x_1, \x_2,s)$, respectively. Even though there is no sharp boundary between the inner and outer regions, it is convenient to introduce an intermediate radius $\delta$, with $\epsilon \ll  \delta \ll 1$, which divides the regions. Then the inner region is $\Omega_\text{in}(\x_1) = \{ \x_2 \in \Omega : \| \x_2 - \x_1 \| < \delta \} $ and the outer region is the complimentary set $\Omega_\text{out}(\x_1) = \Omega \backslash \Omega_\text{in} (\x_1)$. Then 
\begin{align} \label{integral_splitP}
p(\x_1)  = \int_{\Omega} P_2(\x_1, \x_2) \, \ud \x_2  = \int_{\Omega_\text{out}}  P_2   \, \ud \x_2 +  \int_{\Omega_\text{in}}  P_2   \, \ud \x_2. 
\end{align}
The outer integral is, using \eqref{outer_ansatz},
\begin{align*}
\begin{aligned}
\int_{\Omega_\text{out}}  P_2  \, \ud \x_2 &= \int_{\Omega_\text{out}}  \Pout  \, \ud \x_2 = q(\x_1)  \int_{\Omega_\text{out}} q(\x_2)  \, \ud \x_2 + O(\epsilon^l)  \\
&= q(\x_1)  \left [ \int_\Omega q(\x_2)  \, \ud \x_2  -   q(\x_1) \delta^dV_d(1)  \right] + O(\epsilon^l,\delta^{d+1}) ,
\end{aligned}
\end{align*}
where $l>d$ is the decay rate of $u$ at infinity (Assumption \ref{assum_u}) and $V_d(1) $ denotes the volume of the unit ball in $\mathbb R^d$. The inner integral is, using the leading-order of \eqref{Pinnersol}, 
\begin{align*}
	\int_{\Omega_\text{in}} &  P_2  \, \ud \x_2 =  \epsilon^d\int_{\|\tilde \x \| \le \delta/\epsilon}    \tilde P \,  \ud \tilde \x  =  \epsilon^d q^2(\x_1) \int_{\|\tilde \x \| \le \delta/\epsilon}     e^{-u (\tilde \x)}   \ud \tilde \x + O(\epsilon \delta^d).
\end{align*}
Combining the two integrals and choosing $\delta$ such that $\delta^{d+1} = \epsilon^l$ with $d<l<d+1$, we obtain
\begin{align*}
p(\x_1) &= q(\x_1)  \left [ \int_\Omega q(\x_2)  \, \ud \x_2  -   q(\x_1) \left(  \delta^dV_d(1)   - \epsilon^d \int_{\|\tilde \x \| \le \delta/\epsilon}    e^{-u (\tilde \x)}  \ud \tilde \x \right) \right] + O(\epsilon^l)\\
	& = q(\x_1)\left [ \int_\Omega q(\x_2)  \, \ud \x_2  - \epsilon^d   q(\x_1)  \int_{\|\tilde \x \| \le \delta/\epsilon}  \left( 1-  e^{-u (\tilde \x)}  \right) \ud \tilde \x  \right] + O(\epsilon^l),
\end{align*}
using that $\delta^d V_d(1) = \epsilon^d V_d(\delta/\epsilon)$. Since $1 - e^{-u(  \tilde \x )}$ decays at infinity, we can extend the domain of integration to the entire $\mathbb R^d$ introducing only exponentially small errors. Therefore, as required,
\begin{align*}
p(\x_1)  = q(\x_1) \left [ \int_\Omega q(\x_2)  \, \ud \x_2  - \alpha_u \epsilon^d q(\x_1)  \right] + O(\epsilon^l),
\end{align*}
where $\alpha_u$ is given in \eqref{alphaV}. To obtain the asymptotic value of the mass of $q$, we integrate the equation above to impose the normalization condition on $p$:
\begin{align*}
	1 = \int_\Omega p \, \ud \x = \left( \int_\Omega q  \, \ud \x\right)^2 - \alpha_u   \epsilon^d \int_\Omega q^2 \, \ud \x + O(\epsilon^l).
\end{align*}
Rearranging, we find that
\begin{equation*}
	\int_\Omega q(\x)  \, \ud \x = 1 + \frac{1}{2} \alpha_u   \epsilon^d  \int_\Omega q^2 (\x) \,\ud \x + O(\epsilon^l) = 1 + a   \epsilon^d + O(\epsilon^l),
\end{equation*}
as required. The constant is $a = \frac{1}{2} \alpha_u  \int_\Omega q^2 (\x) \,\ud \x$.
\end{proof}

Now we turn to the main contribution of this work. 

\begin{proof}[Proof of Proposition \ref{prop:N2}]
	We consider now the integrated equation \eqref{integ1} with $N=2$ 
\begin{equation}
\label{integ-s}
\frac{\partial p}{\partial s}  + \nabla_{\x_1} \cdot \int_{\Omega}  P_2  \nabla_{\x_1} \Phi_2  \, \ud \x_2  = 0,
\end{equation}
and $\Phi_2$ defined in \eqref{phi_truncated2}.
We introduce a macroscopic mobility $m$ and a macroscopic flow $\phi$ such that 
\begin{equation}
\label{macroflow_def}
m  \nabla_{\x_1} \phi = \int_{\Omega}  P_2 \nabla_{\x_1} \Phi_2  \, \ud \x_2=: \mathcal I (\x_1,s). 
\end{equation}
The integral $\mathcal I$ can be evaluated splitting it into inner and outer parts as in Lemma \ref{lem:rel_pq}. 
Using \eqref{outer_ansatz}, the outer component 
\begin{align}\label{Iout_N2} 
\int_{\Omega_\text{out}}  P_2 \nabla_{\x_1} \Phi_2  \, \ud \x_2 = q(\x_1) \nabla_{\x_1} \varphi(\x_1) \left [ \int_\Omega q(\x_2)  \, \ud \x_2  - q(\x_1) \delta^dV_d(1)) \right] + O(\delta^{d+1}),
	\end{align}
where $\delta \gg 1$ as before, whereas the inner-region integral reads, using \eqref{innersol}, 
\begin{align} \label{Iin_N2}
\begin{aligned}
 \int_{\Omega_\text{in}}  & P_2 \nabla_{\x_1} \Phi_2  \, \ud \x_2 =  \epsilon^{d-1}\int_{\|\tilde \x \| \le \delta/\epsilon}    \tilde P \left( \epsilon \nabla_{\tilde \x_1} - \nabla_{\tilde \x} \right) \tilde \Phi  \ud \tilde \x \\
	&=  \epsilon^{d-1}\int_{\|\tilde \x \| \le \delta/\epsilon}    \Big[ -\tilde P^{(0)}  \nabla_{\tilde \x} \tilde \Phi^{(0)}  \\
	& \hspace{2.6cm} + \epsilon  \big( \tilde P^{(0)} \nabla_{\tilde \x_1} \tilde \Phi^{(0)} - \tilde P^{(1)} \nabla_{\tilde \x} \tilde \Phi^{(0)} - \tilde P^{(0)} \nabla_{\tilde \x} \tilde \Phi^{(1)}\big)  + O(\epsilon^2)\Big] \ud \tilde \x \\
	&= \epsilon^d q^2(\tilde \x_1) \nabla_{\tilde \x_1}  \varphi(\tilde \x_1) \int_{\|\tilde \x \| \le \delta/\epsilon}    e^{-u (\tilde \x)}  \ud \tilde \x + O(\epsilon \delta^{d}).
	\end{aligned}
\end{align}
Combining the two integrals \eqref{Iout_N2}-\eqref{Iin_N2} as in Lemma \ref{lem:rel_pq} and using \eqref{relation_p_vq}, we obtain
\begin{align*}
	\mathcal I (\x_1) &= q(\x_1) \nabla_{\x_1} \varphi(\x_1) \left [ \int_\Omega q(\x_2)  \, \ud \x_2  - \alpha_u \epsilon^d q(\x_1)  \right] + O(\delta^{d+1})\\ &= p(\x_1) \nabla_{\x_1} \varphi(\x_1) + O(\epsilon^l),
\end{align*}
choosing $\delta$ such that $\delta^{d+1} = \epsilon^l$ with $d<l<d+1$ as in Lemma \ref{lem:rel_pq}.
Therefore, from \eqref{macroflow_def} we have that
\begin{equation} \label{rel_varphi_phi}
	m  \nabla_{\x_1} \phi = p \nabla_{\x_1} \varphi + O(\epsilon^l), \quad l>d.
\end{equation}
We use this expression and the condition at time $\Delta t$ to determine the mobility $m $ generally. Using the condition  \eqref{gs_fsb} on $\varphi$ and \eqref{relation_p_vq} we have that, to $O(\epsilon^d)$,
\begin{align*}
m  \nabla_{\x_1} \phi &= -p \nabla_{\x_1} (\log q + V) = -p \nabla_{\x_1} \left[ \log \frac{p}{1 + \epsilon^d(a- \alpha_u q)}  + V\right]\\
&= -p \nabla_{\x_1} \left[ \log \frac{p}{1 + \epsilon^d(a- \alpha_u p)}  + V\right]. 
\end{align*}
Expanding the mobility and the potential in powers of $\epsilon^d$, $m = m^{(0)} + \epsilon^d m^{(1)} + \cdots$ and $\phi = \phi^{(0)} + \epsilon^d \phi^{(1)} + \cdots$, this implies that, at leading order,
\begin{equation*}
	m^{(0)} \nabla_{\x_1} \phi^{(0)} = -  p\nabla_{\x_1} (\log p  + V) 
\end{equation*}
Since the leading order term coincides with case of non-interacting particle, we need to have consistency with the Wasserstein metric (Assumption \ref{assum_flow}). This requires
   $\nabla_{\x_1} \phi^{(0)} = \nabla_{\x_1} \varphi(p)$, where $\varphi(q) = -\log q -V$ and hence  $m^{(0)}(p) = p$. 
At the next order we have that, expanding the logarithm,
\begin{equation} \label{ord1_cond}
m^{(0)} \nabla_{\x_1} \phi^{(1)} + m^{(1)} \nabla_{\x_1} \phi^{(0)} =   p\nabla_{\x_1} (a -  \alpha_u p)  = -  \alpha_u p \nabla_{\x_1} p.
\end{equation}
Substituting for $m^{(0)}$ and $\nabla_{\x_1} \phi^{(0)}$  we obtain
$$ p  \frac{\partial \phi^{(1)} }{\partial p} \nabla_{\x_1} p + p  \frac{\partial \phi^{(1)} }{\partial V}  \nabla_{\x_1} V  - M^{(1)}(  \nabla_{\x_1}p + p \nabla_{x_1} V)    = -  \alpha_u p \nabla_{\x_1} p$$
with $M^{(1)} p = m^{(1)}$. Since this relation is to hold for all $V$ and $p$, in particular their gradients being linearly independent we conclude
$$  p \left( \frac{\partial \phi^{(1)} }{\partial p}  + \alpha_u \right)- M^{(1)} = 0, \qquad  \frac{\partial \phi^{(1)} }{\partial V} - M^{(1)} =0. $$
This implies the conservation law
$$ p \frac{\partial  }{\partial p}  \left(\phi^{(1)} + \alpha_u p \right) =  \frac{\partial  }{\partial V}  \left(\phi^{(1)} + \alpha_u p \right) ,$$
with solution 
$$ \phi^{(1)} = F(\log p + V) - \alpha_u p, \qquad m^{(1)} = p  F'(\log p + V)   $$ 
for an arbitrary differentiable function $F$. A particular solution is given by $F=0$, which leads to $\phi^{(1)} = - \alpha_u p$ and $m^{(1)} = 0$. With this choice, the macroscopic flow is, to $O(\epsilon^d)$,
\begin{equation} \label{rel_varphi_phi2}
	\phi =  \varphi(p) - \alpha_u \epsilon^d p   = -\log p - V - \alpha_u \epsilon^d p .
\end{equation}
 
We now put everything together to show that the macroscopic pair $(p, \phi)$ satisfies the optimality conditions \eqref{macro_opti}. The Euler--Lagrange equation \eqref{macro1} and the no-flux boundary condition \eqref{macro3} are satisfied using \eqref{integ-s} and \eqref{macroflow_def} with $m(p) = p$. The flow $\phi$ satisfies the Hamilton--Jacobi equation \eqref{macro2} up to $O(\epsilon^d)$ using \eqref{micros_leadPhi} together with \eqref{rel_varphi_phi}. The final time condition \eqref{macro5} is exactly given by \eqref{rel_varphi_phi2}.  
\end{proof}

Let us mention that the proof indicates that the choice of $\phi$ is not unique, but our assertion is only that the specific choice $F=0$ yields a solution. The non-uniqueness is related to the fact that also the Fokker--Planck equation \eqref{fps_one_eq} can be written as (here for $N=2$)
$$ \frac{\partial \pe}{\partial t} = \nabla_{\x} \cdot \left \{ \pe [1+\epsilon^d F'(\log \pe + V) ] \nabla_\x \left[ V + \log \pe + \alpha_u \epsilon^d \pe - \epsilon^ d F(\log \pe + V) \right]
\right \}$$
up to terms of order $\epsilon^d$. For each $F$ we obtain a different gradient flow structure; but only the structure for $F$ constant, that is, mobility equal to $p$, has a mobility independent of $V$, which seems a reasonable choice. Note the same ambiguity is apparent in the microscopic Fokker-Planck equation for two particles, we could always rewrite it with a nonlinear mobility depending on the potential, for $F$ nonlinear it would lose the gradient flow structure however. If one accepts the Wasserstein metric as the natural one for the Fokker-Planck equation the ambiguity is eliminated.

Finally, we point out that our Assumptions \ref{assum_u} about the interaction potential $u$ could be refined to potentials such that $\alpha_u$ is defined. 

\subsection{Soft-sphere particles: general \texorpdfstring{$N$}{N}} \label{sec:soft_spheres_limit_N_general}

In this section we outline the result in Corollary \ref{cor:Ncase}.
For a general $N$, we consider a decomposition of the microscopic flow of the form \cite{Burger:2017ob}
\begin{equation}\label{phi_truncated}
	\Phi(\x_1, \dots, \x_N, s) = \sum_{i=1}^N \varphi(\x_i,s) + \sum_{i=1}^N \sum_{j>i}^N \varphi_2(\x_i,\x_j, s) + \cdots.
\end{equation}
In \eqref{phi_truncated}, each new term in the series describes higher-order interactions. For example, for non-interacting particles $\varphi_n = 0$ for all $n\ge2$. However, in the scaling we consider here, we find that $\varphi_n$ vanishes in the outer region for $n \geq 2$ and is of at least of order $\epsilon^2$ for $n \geq 3$ in the inner region. We give more detailed arguments on the expansion of $\Phi$ in the $N$ particle case in Appendix  \ref{app:Ngeneral}.

For non-interacting particles, we have seen that the microscopic density $P$ is the product of $N$ one-particle densities $p$, while the microscopic flow $\Phi$ is the sum of $N$ one-particle flows $\phi$. For pairwise interacting particles, we can neglect the interactions at all orders higher than two in \eqref{phi_truncated} since they lead to higher-order terms. Specifically, starting from the integrated equation \eqref{integ1},  we write
\begin{align}
\label{integ1_N}
\frac{\partial p}{\partial s} + \nabla_{\x_1} \cdot  \mathcal I_N = 0, \qquad \mathcal I_N(\x_1,s) = \int_{\Omega^{N-1}} \left( P \nabla_{\x_1} \Phi \right) \, \ud \x_2 \dots \ud \x_N,
\end{align}
and, as in the $N=2$ case, we define the macroscopic mobility and flux such that $m(p) \phi = \mathcal I_N$. 
This integral over configuration space can be split according to the number of particles from $\x_2, \dots, \x_N$ that are within an order $\epsilon$ distance to $\x_1$:
\begin{equation*}
	\mathcal I_N = \mathcal I_{N,\text{out}} + \sum_{i=2}^N \mathcal I_{N,\text{in}(i)} + \sum_{i_2>i_1\ge 2}^N \mathcal I_{N,\text{in}(i_2,i_2)}^{(2)} + \cdots,
\end{equation*}
where $\mathcal I_\text{out}$ is the integral of the outer expansion over the whole domain $\Omega^{N-1}$, $\mathcal I_{\text{in}(i)}$ is the integral correction over the region corresponding to $\| \x_1 - \x_i\| = O(\epsilon)$. Likewise, $\mathcal I_{\text{in}(i_1,\ldots,i_k)}^{(k)}$ is the integral correction (to the first $k-1$ terms) over the region corresponding to $\x_{i_1}, \ldots \x_{i_k}$ being in the inner region of $\x_1$. 
The corresponding values for the first two terms can be computed using the inner expansion for the $N$-particle case in Appendix \ref{app:Ngeneral} and following the calculation of the proof of Proposition \ref{prop:N2} as 
\begin{align*}
    \mathcal I_{N,\text{out}} &= \int \prod_i q(\x_i)  \nabla \varphi(\x_1)\, \ud \x_2 \dots \ud \x_N = q(\x_1) \nabla \varphi (\x_1) \left(\int_\Omega q(\x_2) \ud\x_2\right)^{N-1} + O(\epsilon^{d+1}),\\
    \mathcal I_{N,\text{in}(i)} &= \epsilon^d \int \left(e^{ u(\tilde \x_i)}-1 \right) \prod_{j=2,j\ne i}^N q(\tilde \x_j)  \nabla \varphi(\tilde \x_1)\, \ud \tilde \x_2 \dots \ud \tilde \x_N  \\
   & = - \epsilon^d \alpha_u q(\x_1 ) \nabla_{\x_1} q(\x_1 ) \left(\int_\Omega q(\x_2)  \ud \x_2\right)^{N-2} + O(\epsilon^{d+1}).
\end{align*}
For the other terms we obtain to leading order 
\begin{align*}
 \mathcal I_{\text{in}(i_1,\ldots,i_k)}^{(k)} &=  \epsilon^{kd} \int \prod_{\ell=2}^k \left(e^{ u(\tilde \x_{i_\ell})}-1 \right) q(\tilde \x_1)^k \prod_j q(\tilde \x_j)  \nabla \varphi(\tilde \x_1)\, \ud \tilde \x_2 \dots \ud \tilde \x_N  =  
O(\epsilon^{kd}),
\end{align*}
where $2\le j \le N$ such that $j \ne i_\ell$, $\ell = 1, \dots, k$ and $\tilde \x_j = \x_j$. 
Thus we see that, for some constant $C$, the estimate
$$\left\vert \sum_k \mathcal I_{\text{in}(i_1,\ldots,i_k)}^{(k)} \right\vert \leq C \sum_{k=2}^{N-1}  \binom{N-1}{k} \epsilon^{kd} 
= O(N^2 \epsilon^{2d}) $$
holds, where we have used the particle indistinguishability and the fact that there are $\binom{N-1}{k}$ tuples $(i_1,\ldots,i_k)$.

Using the above, $\mathcal I_N$ in \eqref{integ1_N} simplifies to
\begin{equation*}
	\mathcal I_N = q(\x_1) \nabla_{\x_1} \varphi(\x_1) \left [ \left(\int_\Omega q(\x_2)\right)^{N-1}  \, \ud \x_2  - (N-1)\alpha_u \epsilon^d q(\x_1)  \right] + O(N \epsilon^l),
\end{equation*}
with $d<l<d+1$. Similarly, the result stated in Lemma \eqref{lem:rel_pq} extends to $N$ general as
\begin{equation*}
	p(\x_1,s) = q(\x_1,s)  \left [ \left( \int_\Omega q(\x,s)  \, \ud \x\right)^{N-1}  - (N-1)\alpha_u   \epsilon^d q(\x_1,s)  \right] + o(N\epsilon^d),
\end{equation*}
Combining the two lines above we arrive again at the same result relating $\varphi$ to the macroscopic flux $\phi$ as in the $N=2$ case, \eqref{rel_varphi_phi}. The remaining steps to arrive at the macroscopic compatibility conditions \eqref{macro_opti} follow exactly the $N=2$ case and will be omitted.  

\subsection{Hard-sphere particles} \label{sec:hs}

In this section we outline the result in Corollary \ref{cor:hs}, extending the result of Proposition \ref{prop:N2} to hard sphere particles. This corresponds to the interaction potential $u_\text{HS}(r/\epsilon)$ with $u_\text{HS}(r) = + \infty$ for $r<1$ and 0 otherwise, so that particles cannot get closer to each other than their diameters $\epsilon$. In this case, it is convenient to move the interaction from the equation (both the SDE \eqref{ssde0} and the microscopic Fokker--Planck \eqref{fps_eqb}) to a reflective boundary condition at $\|\x_1 - \x_2\| = \epsilon$, so that the microscopic problem does not have any singular terms in the equation (see \cite{Bruna:2012cg}). The domain of definition is then given by $\Omega_\epsilon^2 = \Omega^2 \setminus \{ (\x_1, \x_2) \in \Omega^2, \| \x_1- \x_2\| \le \epsilon \}$. The microscopic compatibility conditions are obtained analogously to the soft particles case, to give
\begin{subequations}
\label{micro_opti_hs}
\begin{alignat}{3}
\label{micro_prob_hs}
0 & = \frac{\partial P_2}{\partial s} + \nabla_{\x_1} \cdot \left(P_2 \nabla_{\x_1}  \Phi_2 \right) + \nabla_{\x_2} \cdot \left(P_2 \nabla_{\x_2}  \Phi_2 \right), \qquad & & \text{in } \Omega_\epsilon^2 \times (0,\Delta t), \\
\label{optim2_hs}
0 &=  \frac{\partial \Phi_2}{\partial s} + \frac{1}{2} \| \nabla_{\x_1 } \Phi_2 \| ^2  + \frac{1}{2} \| \nabla_{\x_2 } \Phi_2 \| ^2,   & & \text{in } \Omega_\epsilon^2 \times (0,\Delta t), \\ 
\label{optim4_hs}
0 & = P_2 \nabla_{\x_i}  \Phi_2 \cdot {\bf n} ,   & & \x_i \in  \partial\Omega \times (0,\Delta t),\\
\label{bcinb}
0& = P \left(\nabla_{\x_2}  \Phi - \nabla_{\x_1}  \Phi \right) \cdot {\bf n}_2 , & & \{\| \x_1- \x_2\| = \epsilon\} \times (0,\Delta t), \\
P_2 &= P_{2, k-1}( \x_1, \x_2), & & \text{in } \Omega_\epsilon^2 \times \{0 \}, \\
\label{finalphib}
\Phi_2 & = -\log P_{2,k} - V(\x_1) - V(\x_2), & & \text{in } \Omega_\epsilon^2 \times \{ \Delta t \}.
\end{alignat}
\end{subequations}

Most of the steps in the derivation are analogous to the soft spheres case in Subsection \ref{sec:soft_spheres_limit_N2}. Hence, here we only highlight the key differences arising when considering hard spheres, and we leave the calculation to Appendix \ref{sec:hs_appendix}:
\begin{itemize}
	\item The microscopic model \eqref{micro_opti_hs} is defined in a perforated domain, namely $\Omega_\epsilon^2 = \Omega^2 \setminus \{ (\x_1, \x_2) \in \Omega^2, \| \x_1- \x_2\| \le \epsilon \}$.
	\item At the microscopic level, the interaction between particles appeared as the term $u( (\x_1 - \x_2)/\epsilon)$ in the final-time condition \eqref{optim3-s} for soft spheres. Instead, here it enters in all the conditions through the perforated domain $\Omega_\epsilon^2$, and the additional boundary condition \eqref{bcinb} which ensures conservation of mass.
	\item However, at the macroscopic level we obtain the same structure for hard spheres than soft spheres. The macroscopic compatibility conditions are the same as in Proposition \ref{prop:N2} with $\alpha_{u_\text{HS}} = V_d(1)$ ($2$ for $d=1$, $\pi$ for $d=2$ and $4\pi/3$ for $d=3$). This is consistent with the macroscopic Fokker--Planck model obtained in \cite{Bruna:2012cg}.
	\item The initial condition $P_0(\x_1,\x_2)$ in the outer region  cannot be separable at all orders as in Assumption \ref{assum_initial_refined}. However, the correction due to overlaps in the initial condition $P_0$ scales with the excluded volume, and therefore, up to $l =d$, \eqref{in_outer} applies. 
\end{itemize}

The connection between the cases of hard spheres and particles interacting with a short-range potential $u$ at the macroscopic level motivates the following definition of an effective hard-sphere diameter.

\begin{definition}[Effective hard-sphere diameter] Given a repulsive short-range potential $u(r/\epsilon)$ with range $\epsilon \ll 1$, we define its relative effective hard-sphere diameter $\epsilon_u$ such that $\alpha_u \epsilon_u^d = \alpha_{u_\text{HS}} = V_d(1)$, or equivalently
\begin{equation}
	\epsilon_u^d  = d \int_0^\infty \left[1-\exp(-u) \right] r^{d-1} \ud r.
\end{equation}
This coincides with Rowlinson's concept of an effective hard sphere diameter and can be generalized to attractive-repulsive potentials \cite{Barker:1967gx,Henderson:2010ei}.

\end{definition}

\section{Numerical examples} \label{sec:numerics}

In this section we present several numerical examples of the macroscopic equation \begin{equation}
\label{macro_eq_num}
\frac{\partial \pe}{\partial t} = \nabla_\x \cdot \left \{  \left[ 1 + \alpha_u (N-1) \epsilon^d \pe \right] \nabla_\x  \pe + \nabla_\x  V (\x) \pe \right \}, \qquad \x \in \Omega,
\end{equation}
to illustrate the effect the nonlinear diffusion has on the behaviour of solutions and their convergence to the steady states. Throughout this section we use no-flux boundary conditions on $\partial \Omega$, where $\Omega = [-1/2, 1/2]^d$. We will also compare the predictions of \eqref{macro_eq_num} with stochastic simulations of the corresponding microscopic model (soft or hard spheres depending on the choice of $\alpha_u$). 

\begin{example}[Convergence rate]
We first consider \eqref{macro_eq_num} in one dimension ($d=1$) without external potential (that is, $V=0$). The corresponding unit-mass steady state is $p_{\epsilon,\infty}  = 1$, which is the unique minimizer of the  free-energy 
\begin{equation}
\label{entropy_num}
\Ee ( \pe) = \int_{\Omega} \left[ \pe \log \pe + \frac{1}{2}\alpha_u (N-1) \epsilon^d (\pe)^2 + V(\x) \pe \right]  \, \ud \x.
\end{equation}
In the case without interactions ($\epsilon = 0$), \eqref{macro_eq_num} is simply the diffusion equation, whose solutions approach the steady state with an exponential convergence rate. However, it is not clear what effect has the nonlinearity in the convergence rate. 

We solve \eqref{macro_eq_num} numerically using the finite-volume method presented in \cite{Carrillo:2015cf}, with the initial condition $\pe(\x, 0) = \chi_{[0.2, 0.4]}$ and $N = 100$ hard rods of length $\epsilon = 0.0015$. We also solve the linear case, setting $\epsilon = 0$ in \eqref{macro_eq_num}. The decay of the  free-energy \eqref{entropy_num} for these two cases is shown in Figure \ref{fig:conv} as $\Delta E (t) = \Ee( \pe(x,t)) - \Ee( p_{\epsilon,\infty}(x))$. We observe the expected exponential convergence with rate $r_0 = 2 \lambda_1 = 2 \pi^2$ in the linear case. In the nonlinear case, we also observe an exponential convergence with rate $r_\epsilon > r_0$. 
\def \scc {0.9}
\def \scl {1}
\begin{figure} \label{fig:conv}
\unitlength=1cm
\begin{center}
\psfrag{t}[][][\scl]{$t$} \psfrag{ReE}[][][\scl]{$\Delta E (t)$} 
\psfrag{r1}[][][\scc]{$r_0$}
\psfrag{r0}[][][\scc]{$r_\epsilon$}
\psfrag{r2}[][][\scc]{$\tilde r_\epsilon$}
  \includegraphics[width = 0.6\textwidth]{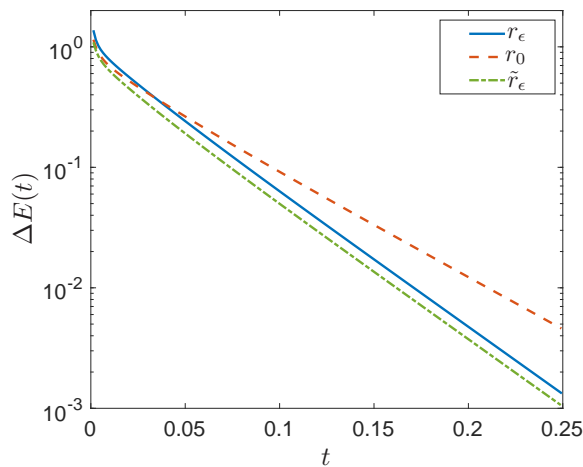} 
  \caption{Free-energy decay towards the equilibrium solution of \eqref{macro_eq_num} with $d=1$, initial data $\pe(\x, 0) = \chi_{[0.2, 0.4]}$, and no external potential ($V=0$). Comparison of the linear case (without interactions, $\epsilon = 0$, red dashed line $r_0$), the nonlinear case with $N = 100$ hard rods  ($\alpha = 2$, blue solid line, $r_\epsilon$) of length $\epsilon = 0.0015$, and the linearized equation around $p_{\epsilon,\infty} = 1$ (green dot-dashed line, $\tilde r_\epsilon$). Slopes of linear fit are shown in the legend.}
  \end{center}
\end{figure}

In order to approximate the increased rate $r_\epsilon$, we look at the linearized version of equation \eqref{macro_eq_num} around the equilibrium $p_{\epsilon,\infty} = 1$, which corresponds to the linear diffusion equation but with diffusion coefficient $D_\text{eff} = 1 + \alpha (N-1) \epsilon ^d$, equal to 1.3 for our choice of parameters. It can easily be shown that in this case the rate is $\tilde r_\epsilon = D_\text{eff} r_0$. We find that the  free-energy decay of the linearized equation agrees with this prediction, as well as with the  free-energy decay of the nonlinear equation (see Figure \ref{fig:conv}), indicating that the solution is already very close to the equilibrium and well approximated by the linearized equation.
\end{example}

In the next examples we compare the behaviour of the solutions of \eqref{macro_eq_num} for different interaction types and external potentials, with the corresponding microscopic  particle-level model. For the particle-level simulations, we use the open-source C\texttt{++} library Aboria \cite{Robinson:2017vxa}. The overdamped Langevin equation \eqref{hsde} is integrated using the Euler--Maruyama method and a constant timestep $\Delta t$. 

In order to compare the models at the density level, we perform $R$ independent realizations and output the positions of all $NR$ particles at a set of output time points. A histogram of the positions is calculated and then scaled to produce a discretized density function ($ p_{i}(t) \approx  p(x_i, t)$ in 1d, $ p_{ij}(t) \approx  p (x_i, y_j, t)$ in 2d,\dots) that can be compared with the solution to \eqref{macro_eq_num}. The macroscopic  free-energy \eqref{entropy_num} should approximate well the microscopic  free-energy. For hard spheres this is 
\begin{equation}
\label{microentropy_hs}
\EN(P)  = \int_{\Omega_\epsilon^N}  \left[ P(\vx, t) \log P(\vx,t) + \sum_{i=1}^N V(\x_i) P  \right] \, \ud \vx.
\end{equation}
Without the $P \log P$ term, this would be straightforward using a Monte Carlo integration; however the entropic term make things more complicated. One approach would be to compute an estimate $\hat P(\vx, t)$ of the joint probability density $P$ (using for example a Kernel Density Estimate) and then obtaining an estimate of $\EN$ either by a so-called re-substitution estimate (Monte Carlo integration using the same samples used to obtain $\hat P$) or a splitting data estimate (generating new samples for the Monte Carlo integration) \cite{Beirlant:1997tw}
\begin{equation}\label{entropy_estimation}
\hat {\EN}_R (t) = \frac{1}{R} \sum_{k=1}^R \left( \log \hat P(\X_1^k, \dots, \X_N^k, t) + \sum_{i=1}^N V(\X_i^k)  \right),
\end{equation}
where $\X_i^k$ is the position of the $i$th particles in the $k$th sample at the time of the   free-energy estimate. 
A cheaper alternative approach is to use the discretized density function $ p_{i}$ and compute the approximation to the  free-energy using a discretized version of \eqref{entropy_num}.  The direct estimation of the microscopic  free-energy \eqref{microentropy_hs} using \eqref{entropy_estimation} is out of the scope of this paper, and we are going to use the second approach. 

\begin{example}[Hard-core interacting particles] We consider a two-dimensional system ($d=2$) with $N=1000$ hard-core disks of diameter $\epsilon = 0.01$, and a quadratic external potential in the horizontal direction, $V(\x) = 5 x^2$ (see Figure \ref{fig:case1}(a)). In two dimensions, the hard-core potential has $\alpha = \pi$, and for our choice of parameters the coefficient of the nonlinear term in \eqref{macro_eq_num} is $\alpha (N-1) \epsilon^2 = 0.314$. We choose initial data constant in the vertical direction so that the evolution of \eqref{macro_eq_num} is purely in the $x$-direction. Specifically, we take initial data $p(\x, 0) = \chi_{[0.1, 0.3]} (x)$. In Figure \ref{fig:case1}(b) we plot the early-time evolution, and in Figure \ref{fig:case1}(c) the steady-state solutions. As expected, we find an increased speed of convergence to equilibrium in the case of interactions (see Figure \ref{fig:case1}(d)). We observe good agreement between the PDE solutions and the stochastic simulations. 
\def \scc {0.5}
\def \scl {.7}
\begin{figure}
\unitlength=1cm
\begin{center}
\vspace{3mm}
\psfrag{x}[][][\scl]{$x$} \psfrag{V}[][][\scl]{$V(x)$} \psfrag{t}[][][\scl]{$t$} \psfrag{pi}[r][][\scl]{$ p_{\epsilon,\infty}$} \psfrag{p}[r][][\scl]{$ \pe$} \psfrag{Erel}[][][\scl]{$\Delta E[ \pe]$} \psfrag{t}[b][][\scl]{$t$}
\psfrag{a}[][][\scl]{(a)} \psfrag{b}[][][\scl]{(c)} \psfrag{c}[][][\scl]{(d)} \psfrag{d}[][][\scl]{(b)}
\psfrag{pdeon}[][][\scc]{PDE int.}
\psfrag{pdeof}[][][\scc]{PDE p.}
\psfrag{simon}[][][\scc]{SDE int.}
\psfrag{simof}[][][\scc]{SDE p.}
	\includegraphics[height = .3\textwidth]{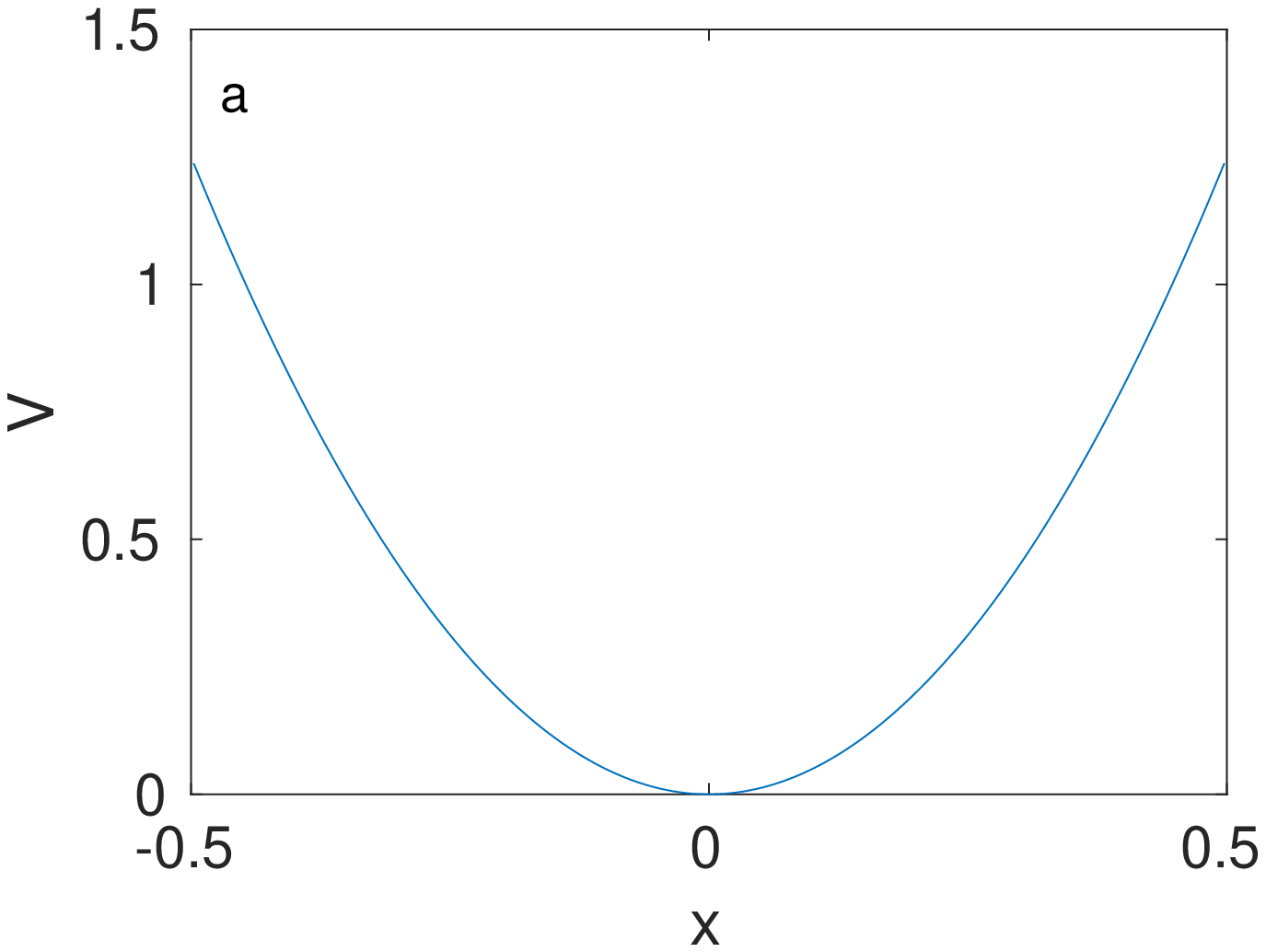} \qquad 
	\includegraphics[height = .3\textwidth]{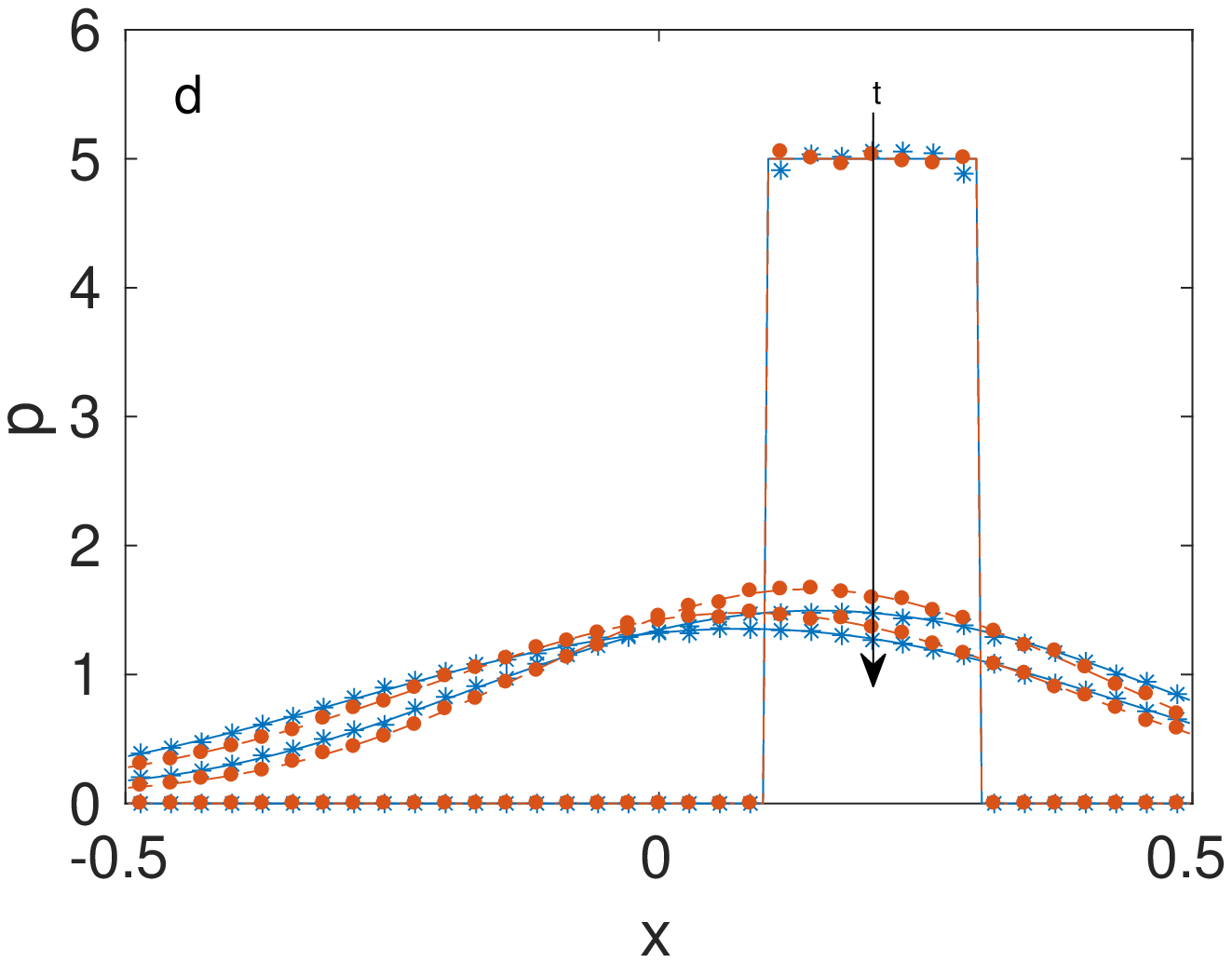}\\ \vspace{3mm}
	\includegraphics[height = .3\textwidth]{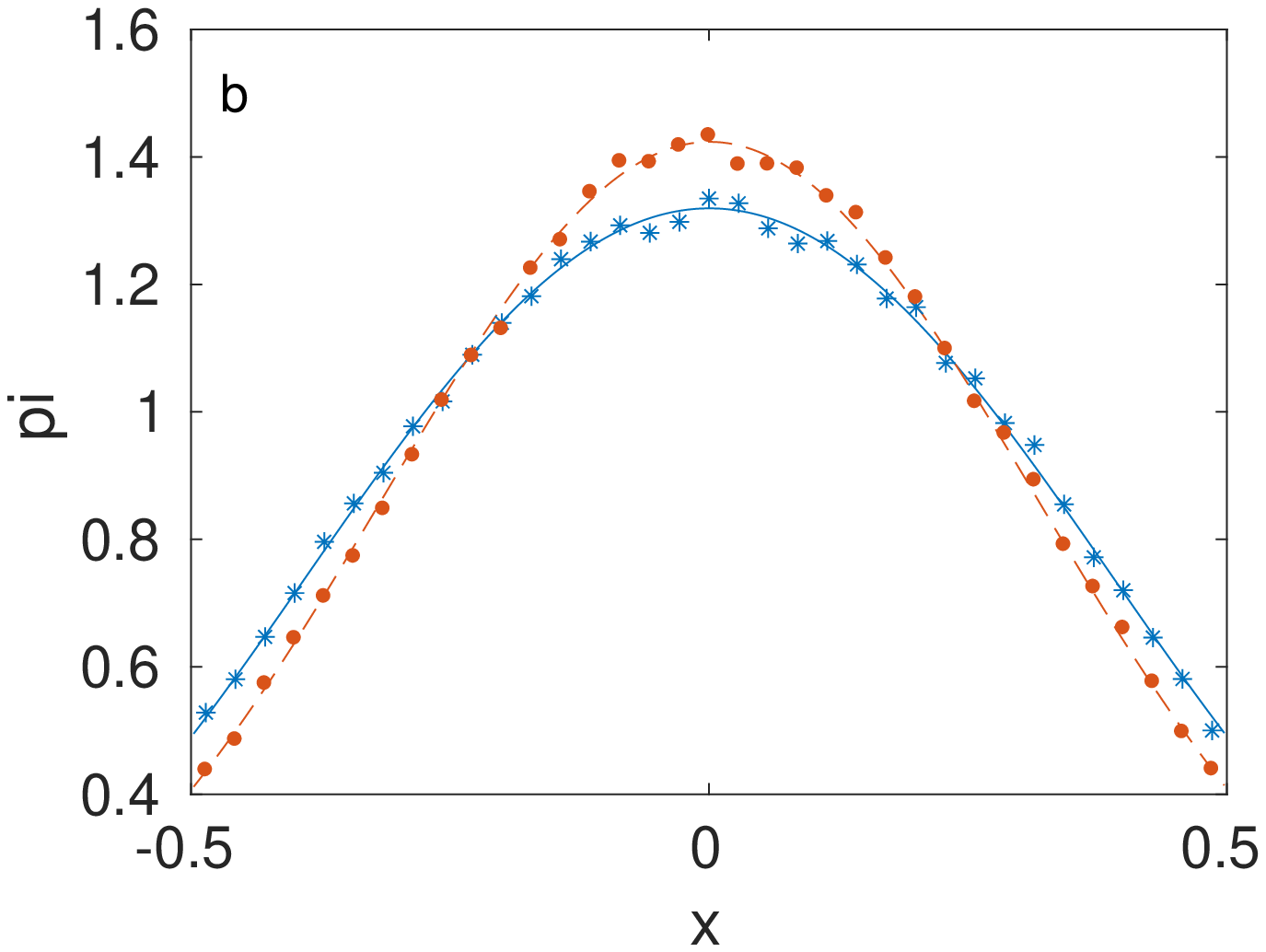} \qquad
	\includegraphics[height = .3\textwidth]{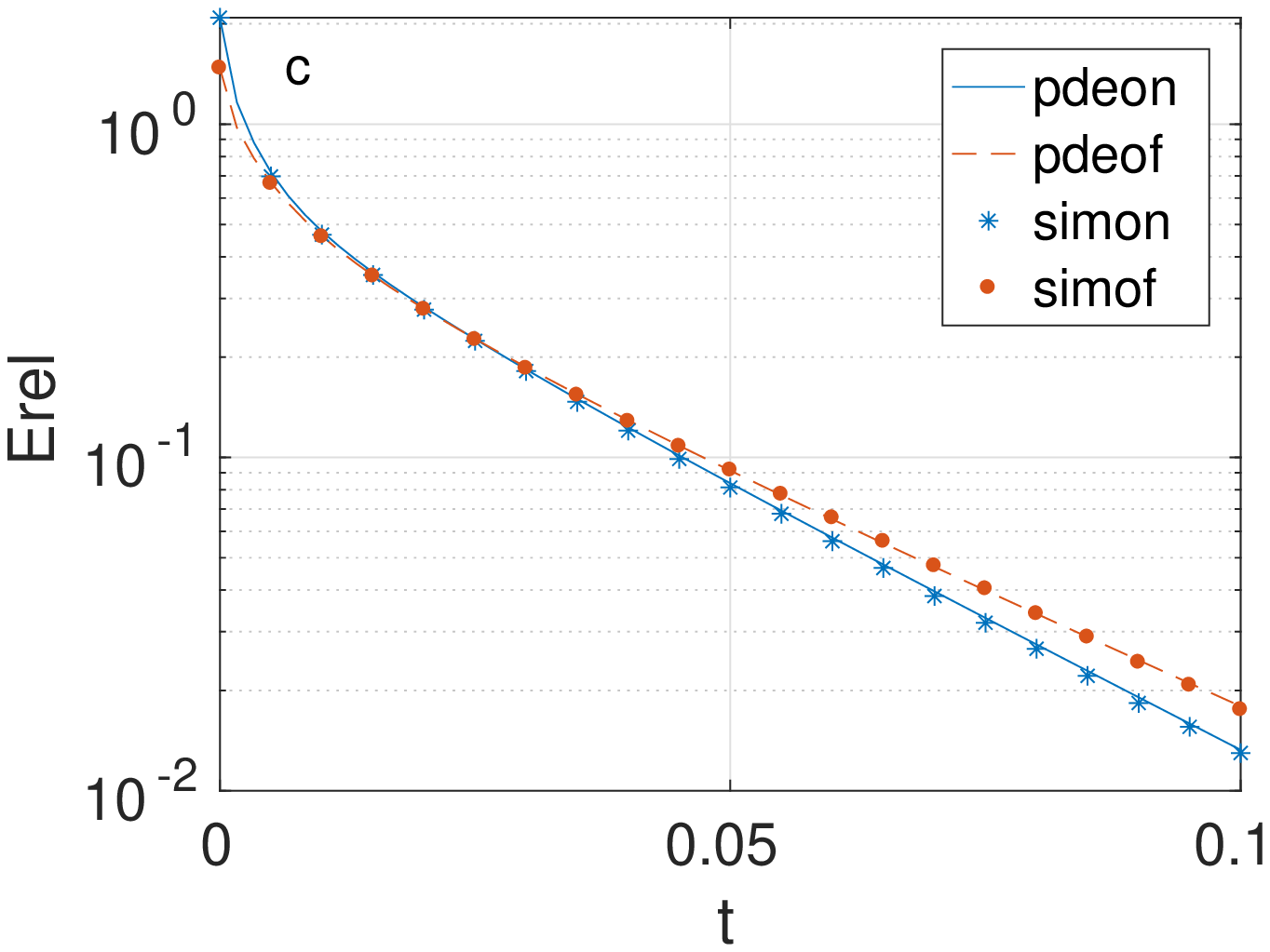} 
\caption{Two-dimensional example with $N=1000$ hard-core particles of diameter $\epsilon = 0.01$ and a quadratic external potential. Comparison between the stochastic simulations of \eqref{hsde} and the solution of the PDE \eqref{macro_eq_num} with and without interactions (corresponding to $\alpha_u = 0$). (a) External potential $V(x) = 5x^2$. (b) Time evolution of the macroscopic density $p_{\epsilon}$ at times $t =0, 0.05, 0.1$. (c) Steady-state $p_{\epsilon,\infty}$. (d) relative entropy  $\Delta E (t) = \Ee( \pe(x,t)) - \Ee( p_{\epsilon,\infty}(x))$. We use $\Delta x = 0.005$ and $\Delta t=10^{-3}$ to solve the PDE and  $R=200$ realizations with $\Delta t= 6.25 \times 10^{-6}$ to generate the histograms.}
\label{fig:case1}
  \end{center}
\end{figure}
\end{example}

\begin{example}[Yukawa interacting particles] In this example, we consider a two dimensional system ($d=2$) with $N=1000$ soft particles with a Yukawa interaction potential, $u(r) =  \exp(-r)/r$ and $\epsilon = 0.01$, and a ``volcano-shaped'' external potential in the horizontal direction (see Figure \ref{fig:case2}(a)). In two dimensions, the Yukawa potential has $\alpha_u = 3.926$ (using \eqref{alphaV}), and for our choice of parameters the coefficient of the nonlinear term in \eqref{macro_eq_num} is $\alpha_u (N-1) \epsilon^2 = 0.392$. We choose initial data to be a sum of two Gaussians along the horizontal direction and constant in the vertical direction so that the evolution of \eqref{macro_eq_num} is purely in the $x$-direction. We observe that the density moves very quickly from the asymmetric initial condition to the centre of the domain, where the minimum of the potential $V$ is (see Figure \ref{fig:case2}(b)). This effect can also be observed in the evolution of the relative entropy, which shows a steep change until around $t = 0.01$ and then relaxes to the long-time convergence (see Figure \ref{fig:case2}(d)). Again we observe a marked difference in the speed of converge to the equilibrium solution between the interacting or point particles simulations, and that the difference in slopes is well-captured by our PDE solutions (see Figure \ref{fig:case2}(d)). 

\def \scc {0.5}
\def \scl {.7}
\begin{figure}
\unitlength=1cm
\begin{center}
\vspace{3mm}
\psfrag{x}[][][\scl]{$x$} \psfrag{V}[][][\scl]{$V(x)$} \psfrag{t}[b][][\scl]{$t$} \psfrag{pi}[r][][\scl]{$ p_{\epsilon,\infty}$} \psfrag{p}[r][][\scl]{$ \pe$} \psfrag{Erel}[][][\scl]{$\Delta E[ \pe]$}
\psfrag{a}[][][\scl]{(a)} \psfrag{b}[][][\scl]{(c)} \psfrag{c}[][][\scl]{(d)} \psfrag{d}[][][\scl]{(b)}
\psfrag{pdeon}[][][\scc]{PDE int.}
\psfrag{pdeof}[][][\scc]{PDE p.}
\psfrag{simon}[][][\scc]{SDE int.}
\psfrag{simof}[][][\scc]{SDE p.}
	\includegraphics[height = .3\textwidth]{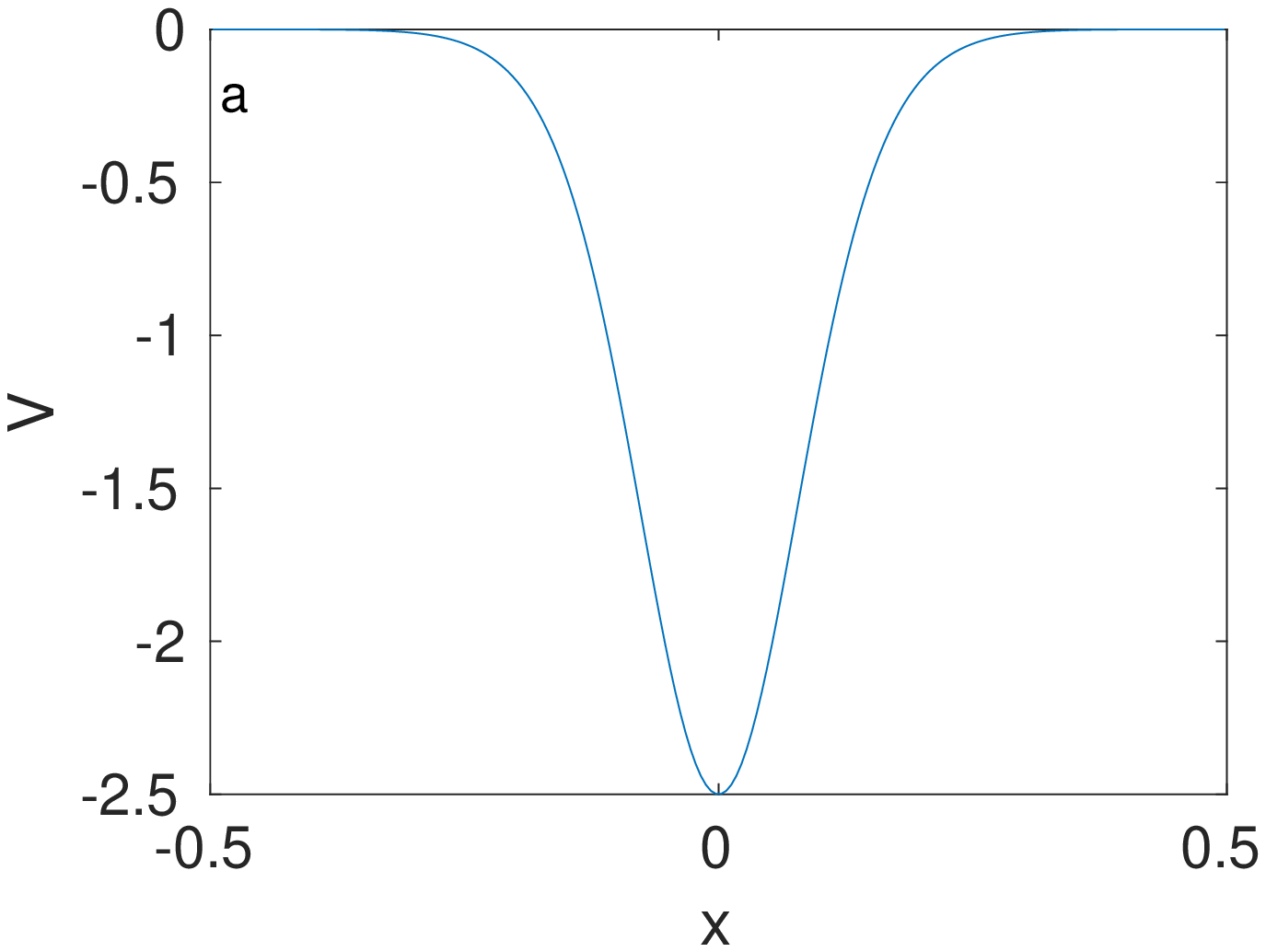} \qquad 
	\includegraphics[height = .3\textwidth]{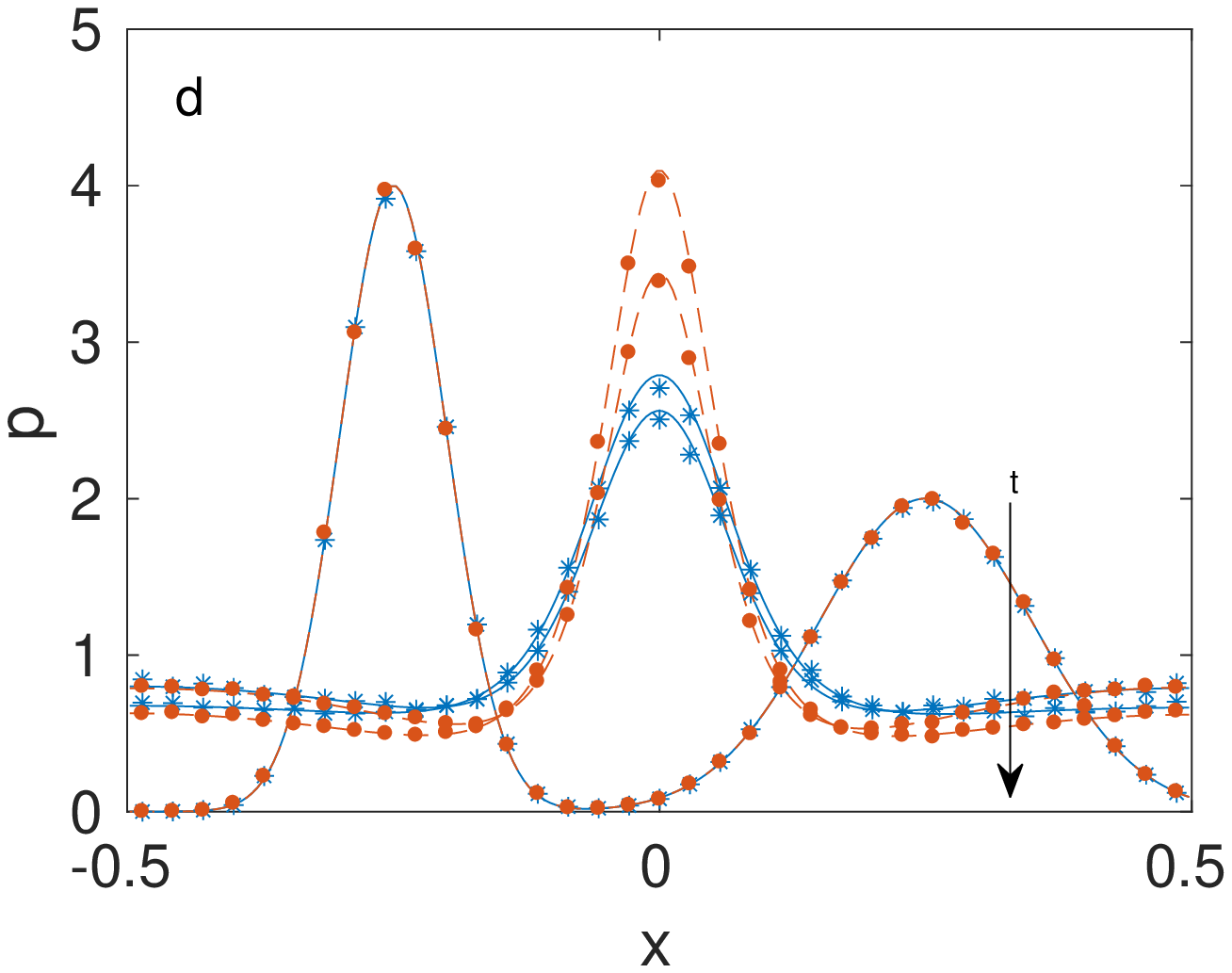}\\ \vspace{3mm}
	\includegraphics[height = .3\textwidth]{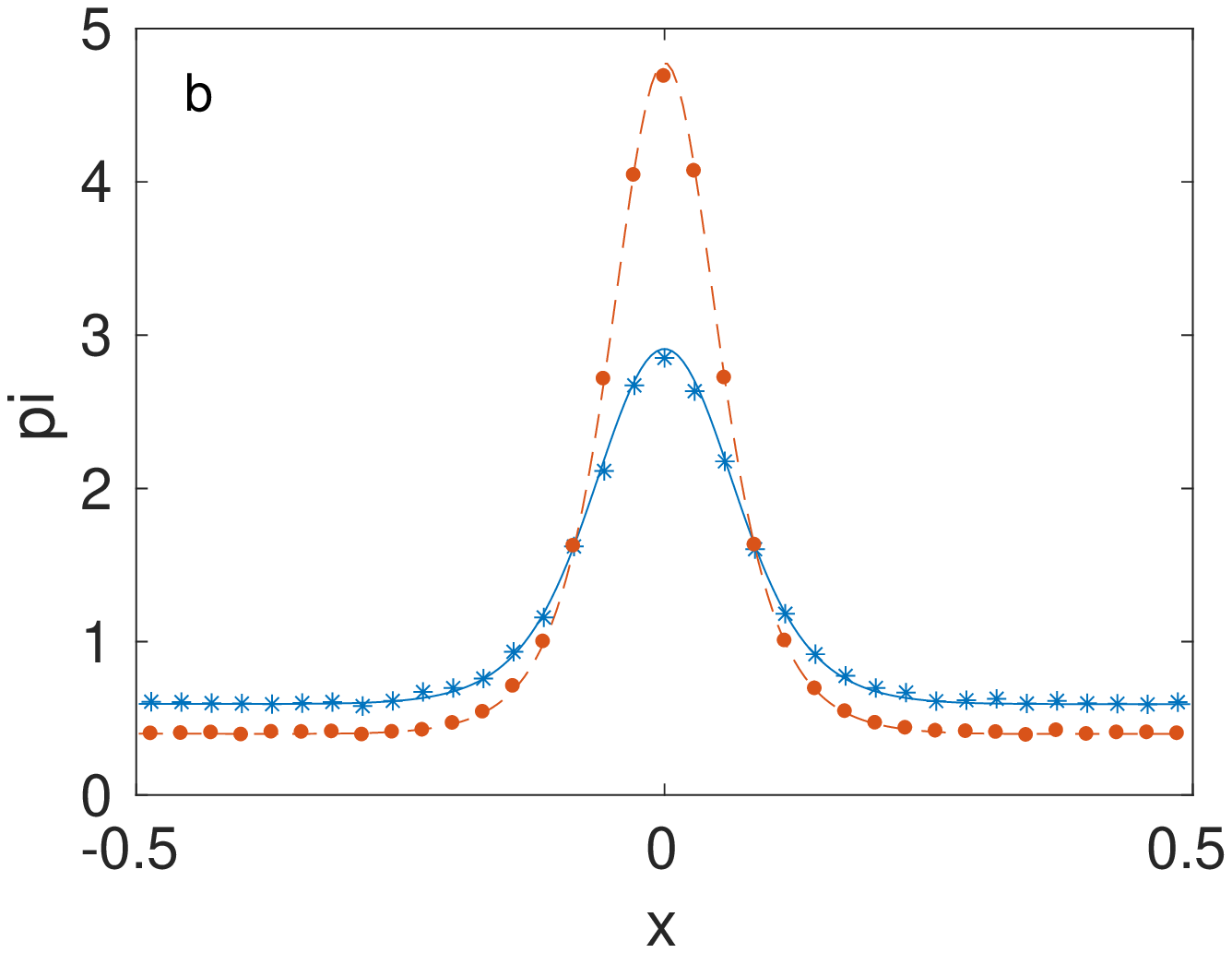} \qquad
	\includegraphics[height = .3\textwidth]{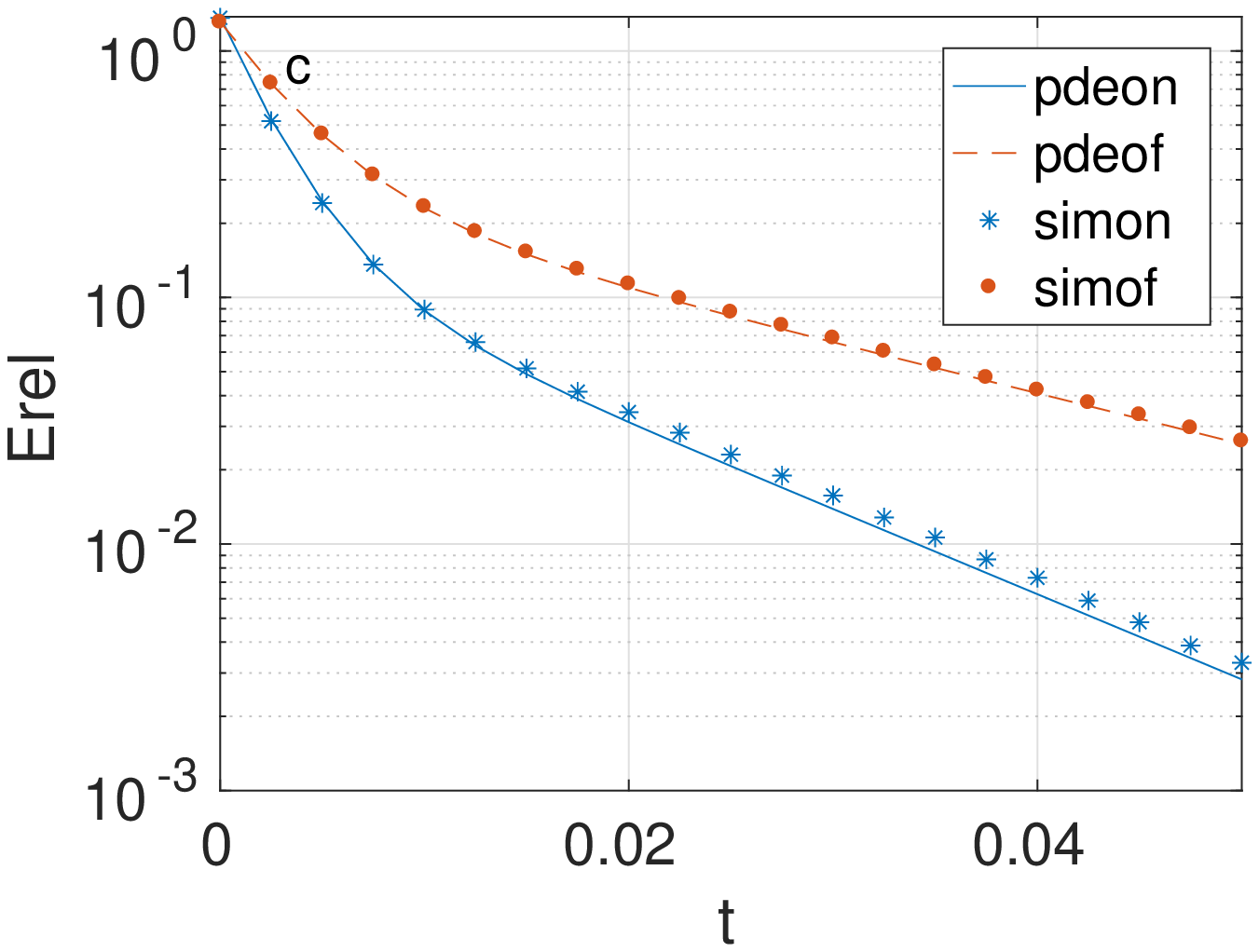} 
\caption{Two-dimensional example with $N=1000$ particles with Yukawa interactions $u(r) =  \exp(-r)/r$, $\epsilon = 0.01$, and a ``volcano-shaped'' external potential. Initial data is a sum of two gaussians, $\pe(0,x) = C [N(-0.25, 0.05^2) + N(0.25, 0.1^2)]$ (with $C$ so that $\pe$ is normalised). (a) External potential $V(x) = -1.5 e^{-x^2/s^2} - e^{-x^2/2s^2}$ (with $s= 0.1$). (b) Time evolution of the solution $\pe$ to \eqref{macro_eq_num} at times $t =0, 0.025, 0.05$. (c) Steady-states $p_{\epsilon,\infty}$ of \eqref{macro_eq_num} with and without interactions (corresponding to $\alpha_u = 0$). (d) Relative entropy  $\Delta E (t) = \Ee( \pe(x,t)) - \Ee( p_{\epsilon,\infty}(x))$.  We use $\Delta x = 0.005$ and $\Delta t=10^{-3}$ to solve the PDE and  $R=200$ realizations with $\Delta t= 2.25 \times 10^{-6}$ ($\Delta t= 6.25 \times 10^{-6}$) to generate the histograms for interacting (point) particles.}
\label{fig:case2}
  \end{center}
\end{figure}

\end{example}

\begin{example}[Power-law interacting particles] In this example, we consider a system with $d=2$, $N=1000$ soft particles with a power-law interaction potential, $u(r) = r^{-4}$ and $\epsilon = 0.01$, and a radially-symmetric ``volcano-shaped'' external potential in the horizontal direction (see Figure \ref{fig:case3}(a)). In two dimensions, this interaction potential has $\alpha_u = 5.568$ (using \eqref{alphaV}), and for our choice of parameters the coefficient of the nonlinear term in \eqref{macro_eq_num} is $\alpha_u (N-1) \epsilon^2 = 0.556$.  This time we choose a radial initial condition whose amplitude depends on the angular variable (see Figure \ref{fig:case3}(b)) so that the evolution is in two dimensions and has two distinct timescales. In particular, we observe a very fast evolution until about $t = 0.01$ by when the mass is almost centred in the middle of the domain (Figure \ref{fig:case3}(c)). After that time the evolution is a lot slower, as seen by the change in slope in relative entropy (Figure \ref{fig:case3}(d)). 
\def \scc {0.5}
\def \scl {.7}
\begin{figure}
\unitlength=1cm
\begin{center}
\psfrag{y}[][][\scl]{$y$} \psfrag{V}[][][\scl]{$V(x)$} \psfrag{t}[][][\scl]{$t$} \psfrag{pi}[r][][\scl]{$ p_{\epsilon,\infty}$} \psfrag{p}[r][][\scl]{$ \pe$} \psfrag{Erel}[][][\scl]{$\Delta E[ \pe]$}
\psfrag{pdeon}[][][\scc]{PDE int.}
\psfrag{pdeof}[][][\scc]{PDE p.}
\psfrag{simon}[][][\scc]{SDE int.}
\psfrag{simof}[][][\scc]{SDE p.}
\psfrag{a}[][][\scl]{(a)} \psfrag{b}[][][\scl]{(b)} \psfrag{c}[][][\scl]{(c)} \psfrag{d}[][][\scl]{(d)}
	\includegraphics[height = .3\textwidth]{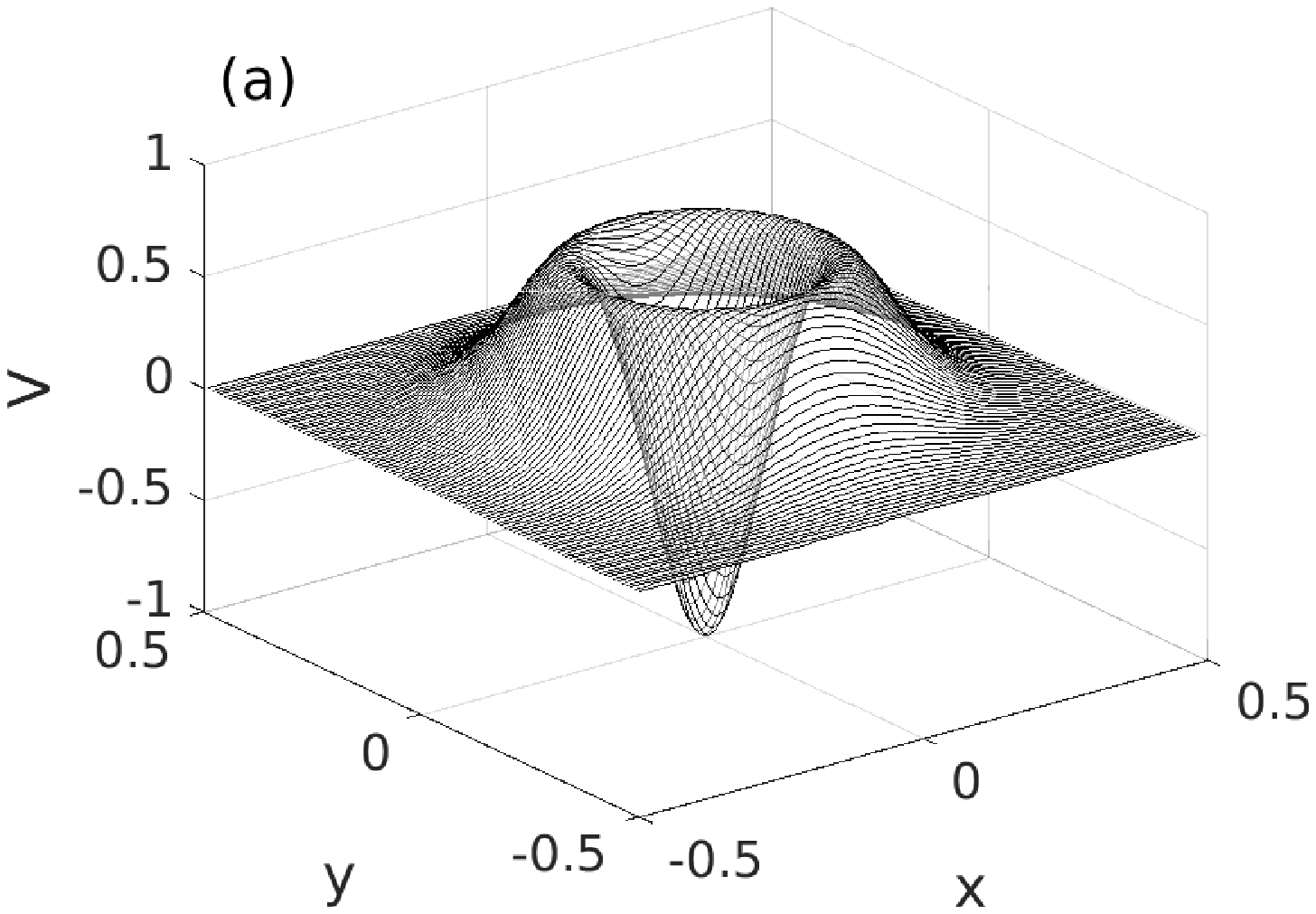} \qquad 
	\includegraphics[height = .3\textwidth]{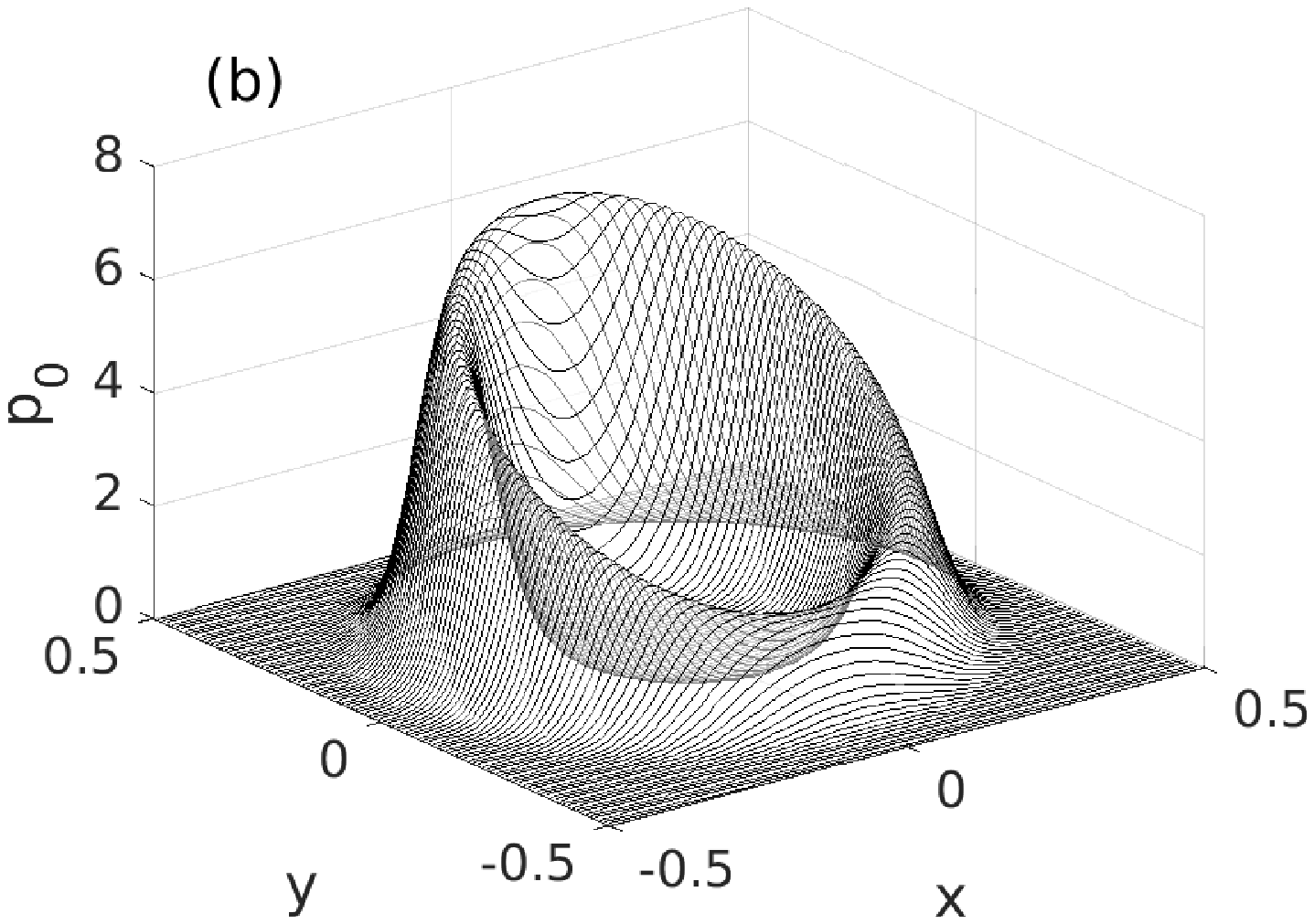} \\ \vspace{3mm}
	\includegraphics[height = .3\textwidth]{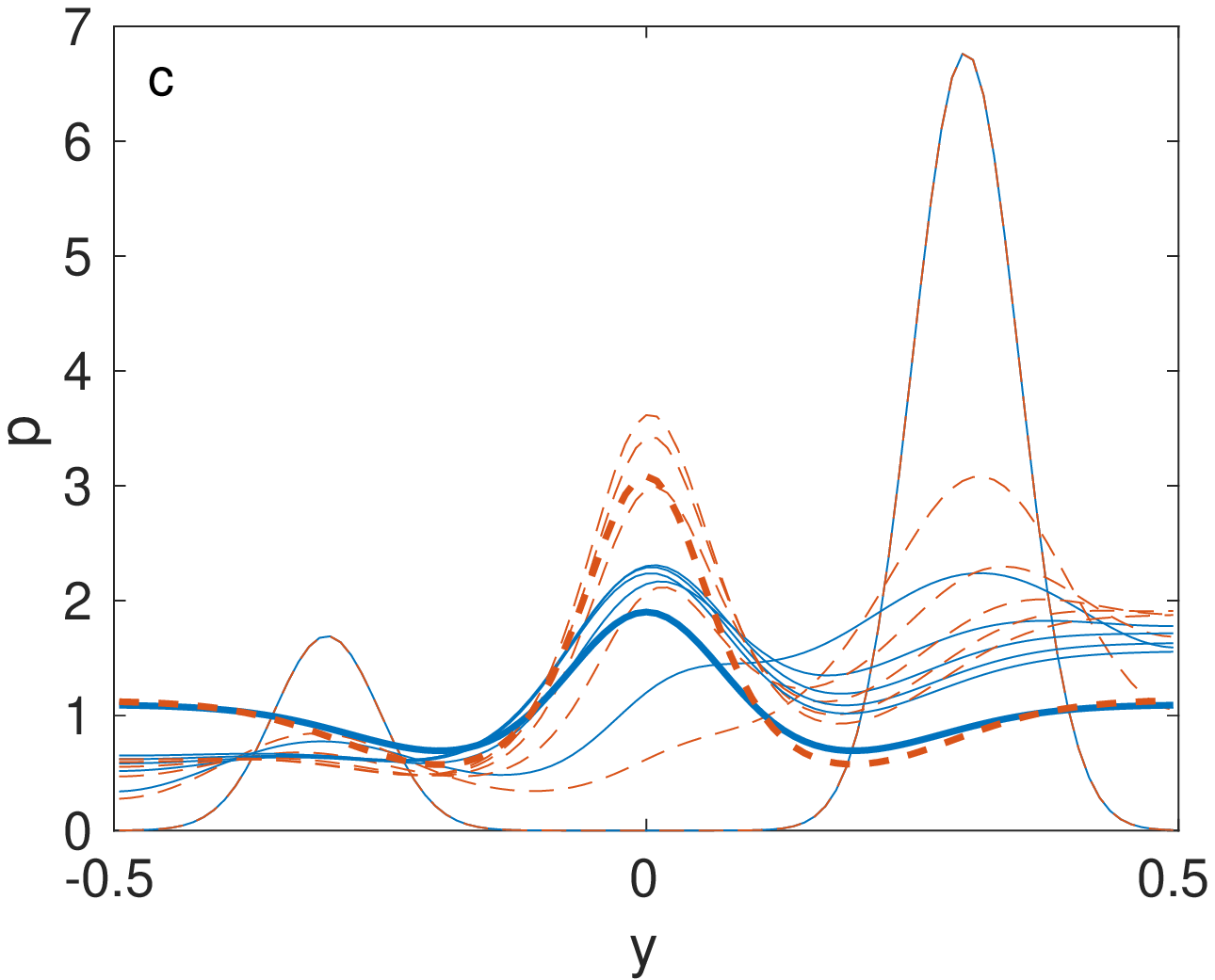}\qquad
	\includegraphics[height = .3\textwidth]{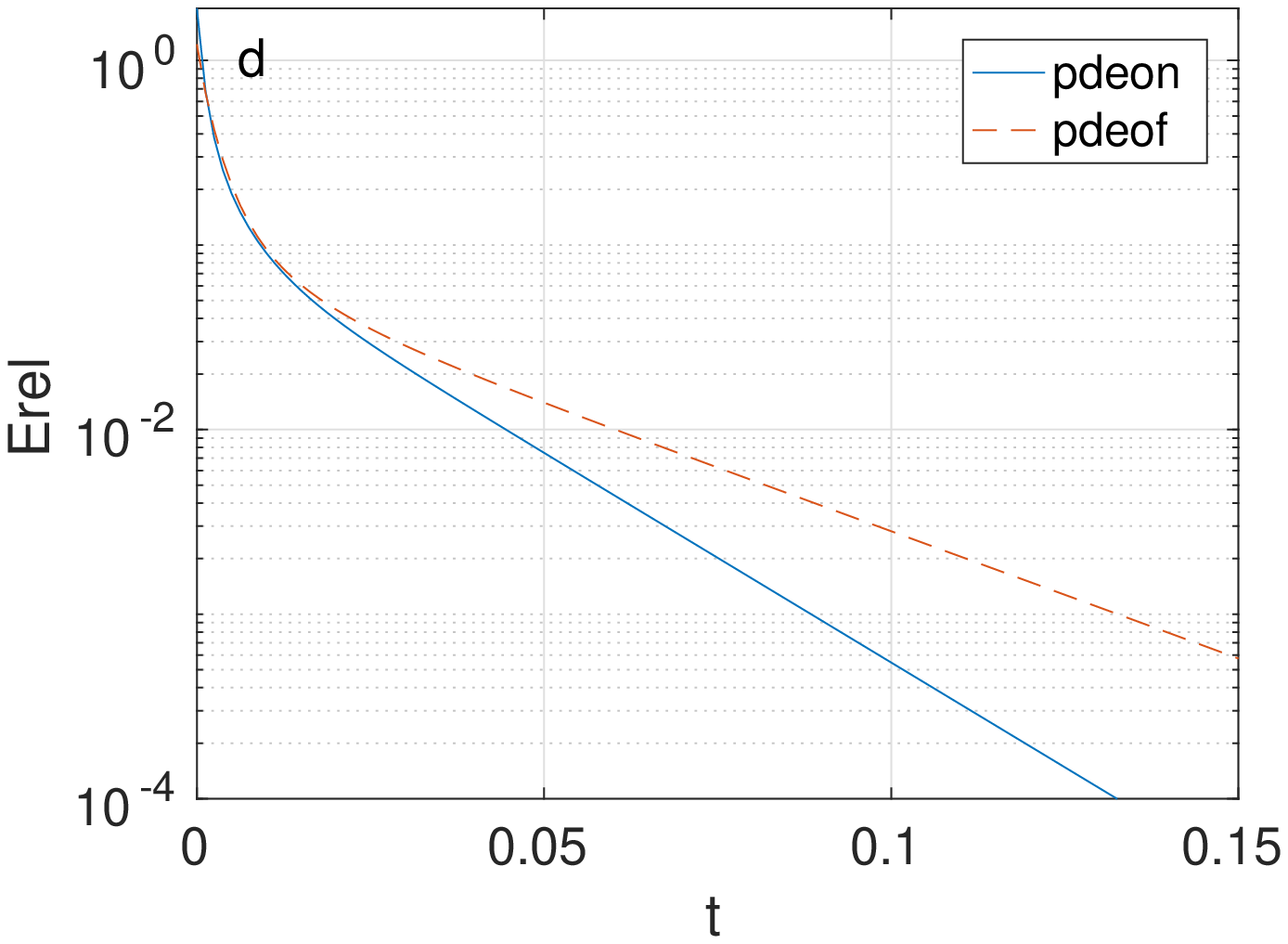} 
\caption{Two-dimensional example with $N=1000$ particles with a power-law interaction potential $u(r) = r^{-4}$, $\epsilon = 0.01$, and a ``volcano-shaped'' external potential. (a) External potential $V(\x) = -4.5 e^{-2s\|\x\|^2} +3.5 e^{-s\|\x\|^2}$ (with $s= 25$). (b) Initial condition $p_0 (r, \theta) = C(1 + 0.6 \sin \theta ) e^{-(r-\mu)^2/2 \sigma^2}$, with $x= r\cos \theta$, $y=r \sin \theta$, $\mu = 0.3$, $\sigma = 0.05$ and $C$ the normalization constant. 
(c) Time evolution of the solution $\pe$ to \eqref{macro_eq_num} along $x=0$. Thin lines correspond to times $t = 0, 0.0025, \dots, 0.0125$; thick lines to the stationary solutions. (d) Relative entropy  $\Delta E (t) = \Ee( \pe(\x,t)) - \Ee( p_{\epsilon,\infty}(\x))$.   
We use $\Delta x = 0.005$ and $\Delta t=10^{-3}$ to solve the PDE and  $R=200$ realizations with $\Delta t= 2.25 \times 10^{-6}$ ($\Delta t= 6.25 \times 10^{-6}$) to generate the histograms for interacting (point) particles.}
\label{fig:case3}
  \end{center}
\end{figure}
\end{example}

\section{Discussion} \label{sec:conclusion}

In this paper we have obtained a framework to derive higher-order expansions of gradient flows for many particle systems, which automatically preserves the gradient flow structure and thus gives a first answer to the questions raised in \cite{Bruna:2016cm}, due to the observation that higher-order expansions only at the level of the PDEs can lead to a loss of the gradient flow structure. It obviously motivates further research, \eg about the cross-diffusion system from \cite{Bruna:2012wu}, where this loss of a gradient structure happens this way \cite{Bruna:2016cm} and we expect to restore the correct gradient structure by our approach.

Let us finally put our asymptotic regime defined in Assumption \ref{assum_regime} in context within the relevant literature on different scaling limits. Oelschl\"ager \cite{Oelschlager:1985kv} considered the model of interacting particles evolving according to \eqref{ssde0} with $V = 0$ in the limit of $N\to \infty$. Bodnar and Velazquez \cite{Bodnar:2005kv} also consider the same model for $d=1$. The interaction is rescaled in the following way as a function of $N$:
$$u(\x) = \frac{1}{N} u_N(\x) = \frac{1}{N} \chi_N^d u_1(\chi_N \x), \quad \chi_N= N^{\beta/d},$$
where $\beta \in [0,1]$ is a parameter that models the strength and range of the interaction. The potential then becomes
\begin{equation} \label{rescaled_u}
	u(\x) = N^{\beta-1} u_1 (\x/l^\beta),
\end{equation}
where we have introduced $l = N^{-1/d}$ (this corresponds to the average distance between uniformly distributed particles in a domain of unit volume). Therefore, the range of the interaction is given by $l^\beta$. In this form, it is easier to extract different cases:
\begin{itemize}
	\item $\beta = 0$: here $u$ scales like $1/N$ and its range is order one. This corresponds to weakly interacting particles (mean field limit), where each particle interacts on average with all the other $N$ particles and the potential is long range. In the limit of $N\to \infty$ one obtains an integro-differential equation.  
	
	\item $\beta \in (0,\beta^*)$: in this case each particle interacts with an order $N^{1-\beta}$ neighbours but each interaction is stronger (of order $N^{\beta-1}$). This limit is termed ``moderately interacting particles'' in \cite{Oelschlager:1985kv}. In contrast to the mean-field case, as $N\to \infty$ the interactions get more and more local and the limit dynamics satisfy a nonlinear diffusion equation of the form \eqref{fps_one_eq} but with a nonlinear coefficient of the form $\int u_1(\x) \ud \x$ (see (28) in \cite{Bodnar:2005kv} or  (15) in \cite{Bruna:2017vr}). In \cite{Oelschlager:1985kv} they have $\beta^* = d/(d+2)$; in \cite{Bodnar:2005kv}, assuming that particles are near the equilibrium (distributed according to the Gibbs measure), they manage to reach $\beta^* = 1$ ($d=1$ only).
	\item $\beta \ge \beta^*$: in contrast to the other two cases, the random fluctuations in the interaction term $\sum_{j>i}^N u( \x_i - \x_j  )$ does not vanish as $N\to \infty$. This corresponds to strongly interacting particles or the hydrodynamic limit. In the case of $\beta = 1$, the strength of $u$ is independent of the number of particles, but the range is short so that on average only interact with on neighbour.
\end{itemize}

What the three limits above have in common is the order of the total excluded volume $\eta$. This is given by $N$ times the volume of the range of the potential, multiplied by the strength of the repulsion (this can be thought of as $\alpha_u$). 
For an interaction potential of the form \eqref{rescaled_u}, we have $\eta = N N^{\beta-1}(l ^{\beta})^d = N^\beta (N^{-1/d})^{\beta d} = 1$, independent of $\beta$.
In contrast to \cite{Bodnar:2005kv, Oelschlager:1985kv}, here we do not take $N \to \infty$, the volume fraction tends to a non-zero small constant, in which we do the asymptotic expansion. In particular, the strength of the potential is order one, its range is $\epsilon$, and the total excluded volume is $N \epsilon^d = \eta$. The assumption $\eta \ll 1$ implies that a particle at a given time only interacts with another particle on average, and that we can use the $N=2$ case to obtain the leading-order correction term. This means that in our work we treat a regime in between the moderately and strongly interacting particles, in that we arrive at a nonlinear diffusion equation for the population density as for $\beta < \beta^*$ but the strength of the interaction is order one.  

Finally, another interesting issue is to study the $N$-particle problem by finding the BBGKY hierarchy equivalent for the $(P, \Phi)$ problem in the sense of \cite{Burger:2017ob} and to provide a justification of truncation in the limit of low volume fraction in this sense alternative to the one provided in Appendix A and Subsection 4.3. 

\begin{acknowledgements}
M. Bruna was partially supported by Royal Society University Research Fellowship (grant number URF/R1/180040).
JAC was partially supported by EPSRC grant number EP/P031587/1 and the Advanced Grant Nonlocal-CPD (Nonlocal PDEs for Complex Particle Dynamics: Phase Transitions, Patterns and Synchronization) of the European Research Council Executive Agency (ERC) under the European Union's Horizon 2020 research and innovation programme (grant agreement No. 883363). 
\end{acknowledgements}

\appendix

\section{The $N$-particle problem} \label{app:Ngeneral}

In the following we briefly discuss the asymptotics in the case of an arbitrary number of $N$ particles (in the soft sphere case), where we have to consider all cases of $k$ particles at distances of order $\epsilon$. We only sketch the arguments, since the detailed computations actually follow closely the case $N=2$.

We fix $\x_1$ and consecutively go through the regions of $k-1$ particles being at distance order $\epsilon$ to $\x_1$; due to indistinguishability we can consider $\x_i$, $i=2,\dots,k$.
The inner region for $k$ particles is determined by a change of variables
\begin{equation*}
	\tilde \x_1 = \x_1, \qquad \tilde \x_i = \frac{\x_i-\x_1}{\epsilon}, \ i=2,\dots,k, \qquad
 \tilde \x_i = \x_i, \ i=k+1,\ldots,N.
\end{equation*}
Similar to \eqref{inner_0s}, the leading order in the Hamilton--Jacobi equation yields $\nabla_{\tilde \x_i} \tilde \Phi^{(0)} = 0$ for $j = 2, \dots, k$, that is, $\tilde \Phi^0$ is independent of the relative distances $\tilde \x_i$ of particles in the inner region. At the next order, similar to \eqref{inner_1-s_eq} 
$$ 0 = \sum_{i=2}^{k} \nabla_{\tilde \x_i} \left[ \tilde P^{(0)} \nabla_{\tilde \x_i} \tilde \Phi^{(1)} -  \tilde P^{(0)} \nabla_{\tilde \x_1} \tilde \varphi \right] $$

The outer region for $k$ particles is determined by one of the $\|\tilde x_i\|$ tending  to infinity, that is, only $k-2$ particles at close distance. Matching the previously obtained solutions for these problems with the inner region is consistent with
\begin{alignat*}{2}
	\tilde P^{(0)} &= e^{-\sum_{i=2}^k  u(\tilde \x_i)} q(\tilde \x_1,s)^k \prod_{j=k+1}^N  q(\tilde \x_j,s), &\qquad \tilde P^{(1)} &= \frac{\tilde P^{(0)}}{ q(\tilde \x_1,s)} \sum_{i=2}^k \tilde \x_i \cdot \nabla_{\tilde \x_1} q(\tilde \x_1,s),\\
	\tilde \Phi^{(0)} &= k \varphi(\tilde \x_1,s)   +  \sum_{j=k+1}^N   \varphi(\tilde \x_j,s), 
& \tilde \Phi^{(1)} &= \sum_{i=2}^k \tilde \x_i \cdot \nabla_{\tilde \x_1} \varphi(\tilde \x_1,s),
\end{alignat*}
again with $\varphi$ and $q$ solving \eqref{micros_lead}. 

Thus, we see that the two leading orders in $\Phi$ depend on particle pairs only, consistent with the leading orders in the two-particle problem.  Hence, when going to the integrated equations the leading order come from the integrals of 
$$\tilde P^{(0)} \nabla_{\x} \tilde \Phi^{(0)} = e^{-\sum_{i=2}^k  u(\tilde \x_i)} q(\tilde \x_1,s)^k \prod_{j=k+1}^N  q(\tilde \x_j,s)\nabla ( k \varphi(\tilde \x_1,s)   +  \sum_{j=k+1}^N   \varphi(\tilde \x_j,s)) . $$ 
Since the integral of these leading-order terms (related to the effective volume in configuration space used for the inner expansions) obtained in the case of $k \geq 3$ particles gives a correction  of negligible order $\epsilon^{2d} N^2$,  the additional terms we obtain in the integration of \eqref{integ1_N} are of higher order than $\epsilon^{2d} N^2$ and can thus be neglected.

\section{Derivation for hard spheres} \label{sec:hs_appendix}

\subsection{Matched asymptotic expansions}

We proceed to solve \eqref{micro_opti_hs} using matched asymptotic expansions. 

In the outer region we obtain the same solution as for soft spheres, as expected since we are outside the interaction region:
\begin{equation}
\label{outerhs_sol}
\begin{aligned}
\Pout ( s , \x_1, \x_2) &= q( s , \x_1) q(  s , \x_2),\\
\Phi_\text{out} ( s , \x_1, \x_2) &= \varphi ( s , \x_1) +  \varphi (  s , \x_2),
\end{aligned}
\end{equation}
where $q$ and $\varphi$ satisfy \eqref{micros_lead}. This is valid up to $O (\epsilon^d)$ by the same argument as in the soft spheres case and the remark that the initial density is chosen separable up to that order (as discussed in Subsection \ref{sec:hs}). 

The inner problem reads
\begin{subequations}
\label{micro_innerb}
\begin{align}
\label{micro_innerb_eq}
0& = \epsilon^2 \frac{\partial \tilde P}{\partial s} +   \nabla_{\tilde \x_1} \cdot (\epsilon^2\tilde P \nabla_{\tilde \x_1} \tilde  \Phi - \epsilon \tilde P \nabla_{\tilde \x} \tilde  \Phi )  +   \nabla_{\tilde \x} \cdot (2 \tilde P \nabla_{\tilde \x} \tilde  \Phi   - \epsilon \tilde P \nabla_{\tilde \x_1} \tilde  \Phi ) ,   \\
\label{micro_flow}
0& = \epsilon^2 \frac{\partial \tilde \Phi}{\partial s} + \frac{\epsilon^2 }{2} \| \nabla_{\tilde \x_1 } \tilde \Phi \| ^2 -  \epsilon   \nabla_{\tilde \x_1 } \tilde \Phi \cdot \nabla_{\tilde \x } \tilde \Phi + \| \nabla_{\tilde \x } \tilde \Phi \| ^2,
\\
\label{initP}
\tilde P(s=0) &= \tilde P_{k-1}( \tilde \x_1, \tilde \x), \\
\label{finalphi}
\tilde \Phi (s=\Delta t) &=  - \left[ \log \tilde P_k( \tilde  \x_1, \tilde \x) + V(\tilde \x_1) + V(\tilde \x_1 + \epsilon \tilde \x) \right], 
\end{align}
together with the boundary condition when two particles are in contact, 
\begin{equation}
	\label{bcinner}
2 \tilde P  \, \tilde \x \cdot \nabla_{\tilde \x} \tilde  \Phi    = \epsilon  \tilde P \,  \tilde \x \cdot \nabla_{\tilde \x_1} \tilde  \Phi     , \qquad \text{on} \qquad \| \tilde \x \| = 1.
\end{equation}
and the matching condition with the outer solution \eqref{outerhs_sol}, with coincides with the soft-sphere condition \eqref{P_match}-\eqref{Phi_match}.
\end{subequations}

Expanding $\tilde P$ and $\tilde \Phi$ in powers of $\epsilon$, $\tilde P \sim \tilde P^{(0)} + \epsilon \tilde P^{(1)} + \cdots$  and $\tilde 
\Phi \sim \tilde \Phi ^{(0)} + \epsilon \tilde \Phi ^{(1)} + \cdots$, the leading order of \eqref{micro_innerb} gives
\begin{subequations}
\label{inner_0b}
\begin{alignat}{2}
\label{inner_0b_1}
 2\nabla_{\tilde \x} \cdot \left (\tilde P^{(0)} \nabla_{\tilde \x} \tilde  \Phi ^{(0)} \right )  &= 0,   &\\ 
 \label{inner_0b_2}
  \left \| \nabla_{\tilde \x } \tilde \Phi ^{(0)}  \right \| ^2  &=0,\\
    \label{inner_0b_3}
\tilde \Phi ^{(0)}(s = \Delta t) &=  - \left[ \log \tilde P^{(0)}_k( \tilde  \x_1, \tilde \x) + 2V(\tilde \x_1) \right], & \\
    \label{inner_0b_6}
2 \tilde P^{(0)} \, \tilde \x \cdot \nabla_{\tilde \x} \tilde  \Phi    &=0 , & \qquad \text{on }  \| \tilde \x \| = 1,\\
\label{inner_0b_4}
\tilde P^{(0)} &\sim q^2(\tilde \x_1,s), & \text{as }  \| \tilde \x \| \to \infty,  \\
\label{inner_0b_5}
\tilde \Phi^{(0)} &\sim 2 \varphi(\tilde \x_1,s), & \text{as }  \| \tilde \x \| \to \infty.
\end{alignat}
\end{subequations}
We find that the behaviour at infinity satisfies in fact all the other constraints. Thus, the leading order is independent of $\tilde s$ and $\tilde \x$:
\begin{equation}
\label{inner-o1b}
\tilde P ^{(0)} = q^2(\tilde \x_1, s), \qquad  \tilde \Phi^{(0)} = 2 \varphi(\tilde \x_1, s).
\end{equation}
The $O(\epsilon)$ of \eqref{micro_innerb} is, using \eqref{inner-o1b},
\begin{subequations}
\label{inner_1b}
\begin{alignat}{2}
\label{inner_1b_eq}
0 &=    2 q^2(\tilde \x_1)  \nabla^2_{\tilde \x}  \tilde  \Phi ^{(1)}, &  \\
\label{final_o1b} 
\tilde \Phi ^{(1)}(s = \Delta t) &=  -    \frac{ \tilde P^{(1)}_k (\tilde \x_1, \tilde \x) }{q_k^2(\tilde \x_1)} -
\tilde \x \cdot \nabla_{\tilde \x_1} V(\tilde \x_1), &\\
\label{inner_1b_bc}
 \tilde \x \cdot \nabla_{\tilde \x} \tilde  \Phi ^{(1)}   &=   \tilde \x \cdot \nabla_{\tilde \x_1} \varphi (\tilde \x_1), &\qquad \text{on }  \| \tilde \x \| = 1,\\
\label{outer_P01b}
\tilde P ^{(1)} &\sim   q(\tilde \x_1)  \tilde \x \cdot \nabla_{\tilde \x_1} q(\tilde \x_1), & \text{as } \| \tilde \x \| \to \infty, \\
\label{outer_phi01}
\tilde \Phi ^{(1)} &\sim  \tilde \x \cdot \nabla_{ \tilde \x_1} \varphi ( \tilde { \bf x}_1), & \text{as } \| \tilde \x \| \to \infty.
\end{alignat}
\end{subequations}
From this we see that $\tilde \Phi ^{(1)} = \tilde \x \cdot \nabla_{ \tilde \x_1} \varphi ( \tilde { \bf x}_1) $ satisfies \eqref{inner_1b_eq}, \eqref{inner_1b_bc} and \eqref{outer_phi01}. Now imposing \eqref{final_o1b} and recalling that $\varphi(\Delta t, \x)  = -\log q_k - V(\x)$ gives
\begin{equation*}
 \tilde P^{(1)}_k (\tilde \x_1, \tilde \x)  =  q_{k}(\tilde \x_1) \tilde \x \cdot \nabla_{\tilde \x_1} q_{k}(\tilde \x_1).
\end{equation*}
Since this satisfies the matching condition \eqref{outer_P01b} at $\Delta t$, and as in the soft-spheres case, the inner problem is stationary up to order $\epsilon$, we can write
\begin{equation*}
 \tilde P^{(1)} (\tilde \x_1, \tilde \x, s)  =  q(\tilde \x_1, s) \tilde \x \cdot \nabla_{\tilde \x_1} q(s,  \tilde \x_1),
\end{equation*}
plus any additional function that vanishes at $s = \Delta t$ and as $\tilde \x \sim \infty$ (but that we can ignore since it does not affect the final integrated result). 
In summary, the solution in the inner region is, to $O(\epsilon^d)$
\begin{subequations}
\label{innersol_hs}
\begin{align}
\label{Pinnersol_hs}
\tilde P (\tilde \x_1, \tilde \x, s) &=q^2(\tilde \x_1, s) + \epsilon  q(\tilde \x_1, s)  \tilde \x \cdot \nabla_{\tilde \x_1} q(\tilde \x_1, s) ,\\
\label{Phiinnersol_hs}
\tilde \Phi (\tilde \x_1, \tilde \x, s)  &= 2 \varphi(\tilde \x_1, s) + \epsilon \tilde \x \cdot \nabla_{ \tilde \x_1} \varphi ( \tilde { \bf x}_1).
\end{align}
\end{subequations}

\subsection{Integrated equations}

The procedure is analogous to the soft spheres case, except that now the domain of integration for $\x_2$ depends on the position of the first particle $\x_1$. This will result in some surface integrals. 
Fixing the first particle at $\x_1$, we integrate \eqref{micro_opti_hs} over the region available to the second particle, namely $\Omega_\epsilon(\x_1) := \Omega \setminus B_\epsilon (\x_1)$. We ignore any intersections that the ball $B_\epsilon (\x_1)$ may have with $\partial \Omega$ (for some positions $\x_1$ close to the boundaries), since this gives a higher-order correction. We find
\begin{equation}
\label{integfz1}
\frac{\partial p}{\partial s}  + \int_{\Omega(\x_1)}  \nabla_{\x_1} \cdot \left( P  \nabla_{\x_1} \Phi \right) \, \ud \x_2 + \int_{\partial B_\epsilon(\x_1)}   P  \nabla_{\x_2} \Phi  \cdot {\bf n}_2 \, \ud S_{\x_2}  = 0,
\end{equation}
where
\[
p(\x_1, s) = \int_{\Omega_\epsilon(\x_1)} P(\x_1, \x_2, s) \, \ud \x_2.
\]
Here $\ud S_{\x_2}$ denotes the surface element with respect to variables $\x_2$. In the last term of \eqref{integfz1} we have used the divergence theorem and the no-flux boundary condition \eqref{optim4_hs}. Using the Reynolds transport theorem in the second term of \eqref{integfz1} and the  no-flux boundary condition \eqref{bcinb} gives
\begin{equation*}
\frac{\partial p}{\partial s}  + \nabla_{\x_1} \cdot \int_{\Omega(\x_1)}  P \nabla_{\x_1} \Phi  \, \ud \x_2  = 0.
\end{equation*}

As in the soft-particles case, we introduce a macroscopic mobility $m$ and a macroscopic flow $\phi$ such that 
\begin{equation*}
m \nabla_{\x_1} \phi = \int_{\Omega(\x_1)}  P \nabla_{\x_1} \Phi  \, \ud \x_2 =: \mathcal I (\x_1, s). 
\end{equation*}
We again compute $\mathcal I$ breaking it into inner and outer regions:
\begin{align*}
\mathcal I (\x_1, s)  = \int_{\Omega_\text{out}({\x}_1)}  P \nabla_{\x_1} \Phi  \, \ud \x_2 + \int_{\Omega_\text{in}({\x}_1)}  P \nabla_{\x_1} \Phi  \, \ud \x_2. 
\end{align*}
The outer part is, using the outer expansion \eqref{outerhs_sol}
\begin{align*}
\int_{\Omega_\text{out}({\x}_1)}  P \nabla_{\x_1} \Phi  \, \ud \x_2 = q(\x_1) \nabla_{\x_1} \varphi(\x_1) \left [ \int_\Omega q(\x_2)  \, \ud \x_2  - q(\x_1) \delta^dV_d(1)  + O(\delta^{d+1}) \right],
	\end{align*}
where $V_d(1) $ denotes the volume of the unit ball in $\mathbb R^d$. The inner region integral, using the inner solution \eqref{innersol_hs}, becomes
\begin{align*}
	\int_{\Omega_\text{in}({\x}_1)}  P \nabla_{\x_1} \Phi  \, \ud \x_2 = (\delta^d - \epsilon^d) V_d(1) q^2(\tilde \x_1) \nabla_{ \x_1}  \varphi( \x_1) + O(\epsilon^{d+1}).
\end{align*}
Combining the two integrals we obtain
\begin{align*}
	\mathcal I &\sim q(\x_1) \nabla_{\x_1} \varphi(\x_1) \left [ \int_\Omega q(\x_2)  \, \ud \x_2  - \alpha_u q(\x_1)  \epsilon^d   \right],
\end{align*}
using that $\alpha_u = V_d(1)$ for hard spheres. Similarly, we can use Lemma \ref{lem:rel_pq} to find $ p  \sim q(\x_1)  \left [ \int_\Omega q(\x_2)  \, \ud \x_2  - \alpha_u\epsilon^d q(\x_1)  \right]$, which implies that $m = p$ and $\phi = \varphi(q)$ up to $O(\epsilon^{d+1})$ as expected.


\label{lastpage}
\end{document}